\documentclass[reqno,oneside]{amsart}
\usepackage[margin=3cm]{geometry}
\usepackage{amssymb,amsmath}
\usepackage{amsfonts}
\usepackage{rotating}
\usepackage[T1]{fontenc}
\usepackage{lmodern}
\usepackage{bbm,bm}
\usepackage{mathtools}
\mathtoolsset{showonlyrefs=true}
\usepackage{graphicx}
\usepackage[boxed,oldcommands,norelsize]{algorithm2e}
\usepackage{hyperref}
\usepackage{subcaption}
\usepackage[numbers,sort&compress]{natbib}
\usepackage{multirow}
\usepackage{hhline}
\usepackage{colortbl}
\newcolumntype{L}{>{\centering\arraybackslash}m{0.46\textwidth}}
\usepackage{arydshln}
\usepackage[nolist]{acronym}

\usepackage{comment}

\newtheorem{lemma}{Lemma}[section]

\newtheorem{definition}{Definition}[section]
\newtheorem{corollary}{Corollary}[section]

\newcommand{\M}{\boldsymbol{\mathsf{M}}}

\newcommand{\sM}{\tilde{\boldsymbol{\mathsf{M}}}}
\newcommand{\oM}{\overline{\boldsymbol{\mathsf{M}}}}

\newcommand{\B}{\boldsymbol{\mathsf{B}}}
\newcommand{\K}{\boldsymbol{\mathsf{K}}}
\newcommand{\sK}{\tilde{\boldsymbol{\mathsf{K}}}}
\newcommand{\oK}{\overline{\boldsymbol{\mathsf{K}}}}
\newcommand{\R}{\boldsymbol{\mathsf{R}}}
\newcommand{\J}{\boldsymbol{\mathsf{J}}}
\newcommand{\0}{\boldsymbol{\mathsf{0}}}
\newcommand{\G}{\boldsymbol{\mathsf{G}}}
\newcommand{\A}{\mathsf{A}}

\newcommand{\U}{\boldsymbol{\mathsf{U}}}

\newcommand{\rr}{\boldsymbol{r}}

\newcommand{\normal}{\boldsymbol{n}}
\newcommand{\x}{\boldsymbol{x}}

\newcommand{\gradient}{\boldsymbol\nabla}
\newcommand{\cc}{\boldsymbol{c}}
\newcommand{\mm}{\boldsymbol{m}}

\newcommand{\half}{\frac{1}{2}}
\def\ij{{ij}} 

\def\i{{\boldsymbol{i}}}

\newcommand{\der}{{\rm d}}
\newcommand{\derd}{\delta}
\newcommand{\dtnp}{\Delta t_{n+1}}
\newcommand{\RR}{\mathbb{R}}

\newcommand{\id}{\mathbb{I}}
\newcommand{\np}{{n+1}}
\newcommand{\nk}{{k,n+1}}
\newcommand{\sr}{\kappa}

\newcommand{\vel}{\boldsymbol{v}}
\newcommand{\conv}{\boldsymbol{f}}

\newcommand{\force}{g}
\newcommand{\bforce}{\boldsymbol{\force}}

\newcommand{\lone}{L^1}
\newcommand{\ltwo}{L^2}

\newcommand{\inflowboundary}{\Gamma_{\rm in}}
\newcommand{\outflowboundary}{\Gamma_{\rm out}}

\newcommand{\domain}{\Omega}

\newcommand{\mesh}{\mathcal{T}_h}

\newcommand{\element}[1][]{K_{#1}}

\newcommand{\nnodes}{N}
\newcommand{\bfespace}{\boldsymbol{V}_h}
\newcommand{\fespace}{V_h}

\newcommand{\nodes}{\mathcal{N}_h}

\newcommand{\neighborhood}[1][]{\mathcal{N}_h(\domain_{#1})}

\newcommand{\symneigh}[1][]{\mathcal{N}^{{#1},\sym}_h}

\newcommand{\contunk}{\boldsymbol{u}}
\newcommand{\unk}{u_h}

\newcommand{\bunk}{\contunk_h}

\newcommand{\test}{v_h}
\newcommand{\btest}{\boldsymbol{v}_h}
\newcommand{\shapef}[1][]{\varphi_{#1}}

\usepackage{stmaryrd}
\newcommand{\smax}{\max{}_{\sigma_h}}

\newcommand{\detector}[1][]{\alpha_{#1}}
\newcommand{\vdetector}[1][]{\boldsymbol{\alpha}_{#1}}

\newcommand{\graphl}{\ell}

\newcommand{\jump}[1]{\left\llbracket #1 \right\rrbracket}
\newcommand{\mean}[1]{%
  \sbox0{%
    \mathsurround=0pt 
    $\left\{\vphantom{#1}\right.\kern-\nulldelimiterspace$%
  }%
  \sbox2{\{}%
  \ifdim\ht0=\ht2
    \{\kern-.625\wd2 \{#1\}\kern-.625\wd2 \}%
  \else
    \left\{\kern-.7\wd0\left\{#1\right\}\kern-.7\wd0\right\}%
  \fi
}

\newcommand{\smthlimit}[1]{Z\left(#1\right)}
\newcommand{\absn}[2][\varepsilon_h]{\left\vert #2 \right\vert_{1,#1}}
\newcommand{\absd}[2][\varepsilon_h]{\left\vert #2 \right\vert_{2,#1}}
\newcommand{\smthdetector}[1][]{ %
	\ifx\empty#1\empty%
	\alpha_{\varepsilon_h} %
	\else %
	\alpha_{\varepsilon_h,#1} %
	\fi}
\newcommand{\vsmthdetector}[1][]{ %
	\ifx\empty#1\empty%
	\boldsymbol{\alpha}_{\varepsilon_h} %
	\else %
	\boldsymbol{\alpha}_{\varepsilon_h,#1} %
	\fi}
\newcommand{\smthspacedetector}[1][]{ %
	\ifx\empty#1\empty%
	\alpha_{\varepsilon_h}^s %
	\else %
	\alpha_{\varepsilon_h,#1}^s %
	\fi}
\newcommand{\smthtimedetector}[1][]{ %
	\ifx\empty#1\empty%
	\alpha_{\varepsilon_h}^t %
	\else %
	\alpha_{\varepsilon_h,#1}^t %
	\fi}
\newcommand{\smthstdetector}[1][]{ %
	\ifx\empty#1\empty%
	\alpha_{\varepsilon_h}^{st} %
	\else %
	\alpha_{\varepsilon_h,#1}^{st} %
	\fi}

\newcommand{\sdetectorAprox}[1][]{ %
 \ifx\empty#1\empty%
    \tilde{\alpha}_{\varepsilon_h} %
 \else %
    \tilde{\alpha}_{\varepsilon_h,#1} %
 \fi}
\newcommand{\artdif}{\nu}
\newcommand{\smthartdif}{\tilde{\nu}}
\newcommand{\sym}{{\rm sym}}
\newcommand{\lambdamax}[1][]{\lambda_{#1}^{\max}}

\graphicspath{{figures/}}

\begin{document}

\begin{acronym}
	\acro{fe}[FE]{finite element}
	\acro{dg}[dG]{discontinuous Galerkin}
	\acro{cg}[cG]{continuous Galerkin}
	\acro{dof}[DOF]{degrees of freedom}
	\acro{ssp}[SSP]{strong stability preserving}
	\acro{rk}[RK]{Runge Kutta}
	\acro{be}[BE]{Backward Euler}
	\acro{dmp}[DMP]{discrete maximum principle}
	\acro{mp}[MP]{maximum principle}
	\acro{afc}[AFC]{algebraic flux correction}
	\acro{fct}[FCT]{flux corrected transport}
	\acro{led}[LED]{local extremum diminishing}
	\acro{dled}[DLED]{discrete local extremum diminishing}
\end{acronym}

\title{On differentiable local bounds preserving stabilization for Euler equations}

\author{ Santiago Badia$^{1,2}$  \and Jes\'us Bonilla$^{2,3}$ \and Sibusiso Mabuza$^{4}$ \and John N. Shadid$^{5,6}$ }
\thanks{
$^1$ School of Mathematics, Monash University, Clayton, Victoria, 3800, Australia. \\ 
\indent$^2$ Centre Internacional de M\`etodes Num\`erics en Enginyeria (CIMNE), Esteve Terradas 5, 08860 Castelldefels, Spain. \\ 
\indent$^3$
Universitat Polit\`ecnica de Catalunya, Jordi Girona 1-3, Edifici C1, 08034
Barcelona, Spain. \\
\indent$^4$ School of Mathematical and Statistical Sciences,
Clemson University, O-110 Martin Hall,
Clemson, SC 29634, USA. \\
\indent$^5$ Center for Computing Research, Sandia National Laboratories, PO Box 5800, Albuquerque NM 87123, USA. \\
\indent$^6$ Department of Mathematics and Statistics, University of New Mexico, MSC01 1115, Albuquerque, NM 87131, USA. \\ 
E-mails:\\ {\tt santiago.badia@monash.edu} (SB), {\tt jbonilla@cimne.upc.edu} (JB), {\tt smabuza@clemson.edu} (SM), {\tt jnshadi@sandia.gov} (JS)
}

\renewcommand{\thefootnote}{\arabic{footnote}}

\maketitle

\begin{abstract} 
This work presents the design 
of nonlinear stabilization techniques for the 
finite element discretization 
of Euler equations in both steady 
and transient form. Implicit time integration 
is used in the case of the transient form. A 
differentiable local bounds preserving method 
has been developed, which combines a Rusanov 
artificial diffusion operator and a differentiable 
shock detector. Nonlinear stabilization schemes 
are usually stiff and highly nonlinear. 
This issue is mitigated by the differentiability
properties of the proposed method. Moreover, 
in order to further improve the nonlinear 
convergence, we also propose a continuation 
method for a subset of the stabilization parameters. 
The resulting method has been successfully 
applied to steady and transient problems with 
complex shock patterns. Numerical experiments 
show that it is able to provide sharp and well 
resolved shocks. The importance of the 
differentiability is assessed by comparing 
the new scheme with its non-differentiable 
counterpart. Numerical experiments suggest that, 
for up to moderate nonlinear tolerances, the method 
exhibits improved robustness and nonlinear 
convergence behavior for steady problems. 
In the case of transient problem, we also 
observe a reduction in the computational cost.
\end{abstract}


\tableofcontents

\pagestyle{myheadings}
\thispagestyle{plain}

\section{Introduction}

The solution of many hyperbolic conservation 
laws satisfies a number of 
mathematical and physics constraints. 
These can include, for example, maximum 
principles, positivity and monotonicity 
preservation. A classical example are the 
Euler equations, where positivity must be 
preserved for the density, internal energy, 
and therefore also the pressure. In general, 
discretizations can yield non-physical 
solutions that violate these properties, 
leading to nonlinear instabilities. 
This is a well known issue. In the context 
of explicit finite volume schemes or \ac{dg} 
\ac{fe} methods, several stabilized schemes 
have already been developed \cite{Leveque2002,Toro2009,cockburn_rungekutta_2001}. 
However, explicit time integrators need to 
resolve all time scales for stability reasons. 
For some multiple-time-scale problems, this 
can often imply very stringent stability 
conditions on the time-step size. If the 
fastest time scales are critical to the 
dynamics, and therefore of scientific or 
engineering interest, then explicit time 
integrators are well suited. On the contrary, 
implicit time integration is favored when 
the smallest time-scales are not relevant 
to the dynamics of interest. For example, 
the ability to integrate accurately and 
efficiently for longer time-scale simulations 
can be essential in some plasma physics 
applications \cite{kritz_fusion_2009}. 
Moreover, explicit schemes become inefficient 
for steady problems because one is forced 
to solve all the hydrodynamic evolution 
until the steady state is reached. 
Therefore, the design of 
implicit \textit{stabilized} schemes 
that preserve the previously mentioned 
structure continues to be an important challenge. 

In this work, we focus on 
implicit \ac{cg} \ac{fe} approximations 
of steady and transient shock hydrodynamics problems. 
It is well known that the Galerkin method 
(without any modification) is generally 
unstable for hyperbolic problems and yields 
solutions with spurious oscillations 
\cite{Kuzmin2005,kuzmin_algebraic_2005}. 
Therefore, \ac{fe} schemes are usually 
supplemented with additional artificial 
diffusion terms. Those terms are designed 
such that the resulting scheme satisfies 
the properties of the continuous problem. 
For example, positive 
density and internal energy or non-decreasing entropy. 
Developing a numerical scheme 
that preserves these properties is very challenging. 
This becomes especially complex for nonlinear 
hyperbolic systems.
A number of methods that preserve the 
continuous problem properties have been 
proposed. In the explicit finite 
difference and finite volume methods contexts, 
schemes in \cite{Hoff1979,Hoff1985,Frid2001} 
preserve these properties for 
Euler and the p-system. Recently, Guermond 
and Popov \cite{Guermond2015} have extended 
these methods to explicit \ac{cg} \ac{fe} schemes. 
Their result is applicable to any first order 
hyperbolic system with bounded wave 
propagation speed. 
In \cite{Guermond2018}, Guermond et al. 
improved their previous scheme to recover 
second order convergence. Moreover, they 
generalized it to different discretization approaches. 
However, this scheme is limited to explicit 
time integration. More recently, Kuzmin 
\cite{Kuzmin2020} has extended this method to 
monolithic convex limiting, which allow the 
usage of implicit time integrators.
An alternative is to impose conditions 
based on the diagonalization of the problem. 
Since it is a hyperbolic system, 
then there exists a set of \emph{characteristic} 
variables for which the system can be 
locally diagonalized and written as a set of 
independent transport problems. At this point, 
one can use techniques developed for scalar 
problems. Hence, the stabilization methods 
are based on adapting the scalar techniques 
for characteristic variables to the system 
written in the original set of variables. 
Following this strategy, some progress has 
been recently made in stabilized \ac{fe} 
schemes by making use of \ac{fct} algorithms 
\cite{Kuzmin2012,Lohmann2016,Mabuza2018,Mabuza2019}. 
The schemes proposed therein are based 
on two main ingredients. On the one hand, 
a diffusive term able to minimize or eliminate 
any oscillatory behavior. On the other hand, 
a limiter, or shock detector, to modulate 
the stabilization term and restrict its 
action to the vicinity of shocks.

With this work focusing on steady and
transient Euler equations, with implicit
time-stepping in the transient case, our
scheme is fully implicit. The nonlinearity
of the Euler equations results in large
nonlinear system of equations requiring
a robust nonlinear solver. We will present
and study a nonlinear solver by looking 
at its nonlinear convergence in addition to
some standard validation.
It is known that for certain limiter choices 
the convergence of the nonlinear solver 
might be remarkably hard 
\cite{kuzmin_linearity-preserving_2012}. 
Recent progress has been made to improve the
convergence of the nonlinear solver
for scalar convection-diffusion problems 
\cite{badia_monotonicity-preserving_2017,
badia_differentiable_2017,Bonilla2019a}. 
In these studies, the authors were able 
to improve the nonlinear convergence by 
proposing differentiable stabilization 
terms to improve convergence rates of the 
Newton iterative nonlinear solver. 
In this work, we extend the 
differentiable nonlinear stabilization in 
\cite{badia_monotonicity-preserving_2017} 
to the Euler equations using the ideas from 
\cite{Kuzmin2012,Mabuza2018} to define the 
artificial diffusion operators for hyperbolic 
systems. The new method is applied to the 
steady and transient Euler equations, 
and its nonlinear convergence is assessed. 
The method proposed in the present work has 
been implemented and tested using 
the \texttt{FEMPAR} library 
\cite{badia_fempar:_2017,Badia2020}.
 
The paper is structured as follows. 
In Sect.~\ref{sec.preliminaries} we 
present the CG discretization for 
Euler equations.
Sect.~\ref{sec.stabilization} is 
devoted to the definition of the 
stabilization terms. We describe 
the nonlinear solvers used in 
Sect.~\ref{sec.solver}. Then, in 
Sect.~\ref{sec.experiments} we present the 
numerical experiments performed. 
Finally, we draw some conclusions in 
Sect.~\ref{sec.conclusions}.

\section{Preliminaries}\label{sec.preliminaries}
In this section, we introduce the continuous 
problem and its \ac{fe} discretization. At the end of the 
section, the basic ingredients for desired 
stabilized schemes for hyperbolic systems are presented.

\subsection{Continuous problem}

Consider an open, bounded, and connected domain, 
$\domain\subset\RR^d$, where $d$ is the number of 
spatial dimensions. Let $\partial\domain$ be the Lipschitz continuous boundary of $\domain$. A first order hyperbolic problem can be written in conservative form as
\begin{equation}\label{eq.continuous-problem}
\arraycolsep=1.4pt\def\arraystretch{1.1}
\left\{\begin{array}{rcll}
\partial_t \contunk + \gradient\cdot\conv(\contunk) & = & 
\bforce, & \text{in } \domain\times(0,T], \\
\contunk(x,0) & = & \contunk_0(x), &  x\in\domain,
\end{array}
\right.
\end{equation}
where $\contunk = \{ u^\beta \}_{\beta=1}^m$ 
are $m\geq 1$ conserved variables, 
$\conv$ is the physical flux, 
$\contunk_0$ are the given initial conditions, 
and $\bforce(x,t)$ is a function defining the body 
forces. Note that the flux, 
$\conv : \RR^m \to \RR^{m\times d}$,  
has components $\conv = \{\conv_i\}_{i=1}^d$, where 
$\conv_i : \RR^m \to \RR^m$ is the flux in 
the $i$th spatial direction. 
Note that if $m=1$, and $\conv(u) \doteq \vel(x,t) u$ 
with $\vel(x,t)$ a divergence-free convection 
field, we recover the well known linear scalar 
convection problem. For Euler equations we
have $m=d+2$ and
\begin{equation}\label{eq.eulervars}
\contunk\doteq \left(\begin{array}{c}
\rho \\ \boldsymbol{m} \\ \rho E
\end{array}\right),\quad \conv\doteq \left(\begin{array}{c}
\boldsymbol{m} \\ \boldsymbol{m}\otimes\vel + p \id \\ \vel(\rho E+p)
\end{array}\right),\, \  \hbox{and} \ \ \bforce\doteq\left(\begin{array}{c}
0 \\ \boldsymbol{b} \\ \boldsymbol{b}\cdot\vel + r
\end{array}\right),
\end{equation}
where $\rho$ is the density, $\rho E$ is the 
total energy, $p$ is the pressure, 
$\boldsymbol{m} = \{{m}_1,\ldots,{m}_d\}$, 
where $m_i = \rho v_i$, is the momentum, 
$\vel = \{v_1,\ldots,v_d\}$ is the velocity, 
$\boldsymbol{b} = \{{b}_1,\ldots,{b}_d\}$ 
are the body forces, $r$ is an energy 
source term per unit mass, and $\mathbb{I}$ 
an identity matrix of dimension $d$. In 
addition, the system is equipped with 
the ideal gas equation of state 
$p=(\gamma-1) \rho e$, where 
$\rho e= \rho E-\half \rho\|\vel\|^2$ is the 
internal energy, and $\gamma$ is the 
adiabatic index. 
We will also consider the steady 
counterpart of \eqref{eq.continuous-problem}, 
which is obtained by dropping the time 
derivative term and the initial conditions.

Boundary conditions for this system may be imposed
strongly or numerically, see 
\cite{Feistauer2003,Toro2009,Gurris2009} 
for details. Here we will present one
approach that makes use of the 
eigensystem of the Euler equations
to define inlet and outlet boundaries.
Denote by $\conv^\prime:\RR^m 
\to \RR^{m\times m\times d}$ 
the flux Jacobian. Let $\normal\in\RR^d$ be 
any direction vector. Since the system is 
hyperbolic the flux Jacobian in any direction 
is diagonalizable and has only real eigenvalues, i.e. $\conv^\prime(\contunk)\cdot\normal=\sum_{i=1}^{d}
\conv_i^\prime(\contunk)n_i $ is diagonalizable 
with real eigenvalues, say $\{\lambda_\beta\}_{\beta=1}^m$. 
These eigenvalues might have different 
multiplicities, and different signs. Hence, 
for a given direction, $\normal$, each 
characteristic variable might be convected 
forward (along $\normal$) or backwards 
(along $-\normal$). Therefore, it is convenient 
to define inflow and outflow boundaries for 
each component. The inflow boundary for 
component $\beta$ is defined as $\inflowboundary^\beta\doteq\{\x\in\partial\domain \ :\ \lambda_\beta(\conv^\prime(\contunk)\cdot
\normal_{{\partial\domain}}) \leq 0\}$, 
where $\normal_{\partial\domain}$ is the 
unit outward normal to the boundary, and 
$\lambda_\beta$ is the $\beta$th-eigenvalue 
of the flux Jacobian. We define the outflow 
boundary as $\outflowboundary^\beta\doteq\partial\domain
\backslash\inflowboundary^\beta$.
From this, we can define some inflow boundary
conditions as follows
$$
u^\beta(x,t) =  \bar{u}^\beta(x,t), \; \text{on } 
\inflowboundary^\beta\times(0,T],\; \beta = 1,...,m,
$$
where $\bar{u}^\beta(x,t)$ are the boundary values for the $\beta$th-component of $\contunk$.

\subsection{Discretization}
Let $\mesh$ be a conforming partition of $\domain$.
The set of nodes of the mesh $\mesh$ 
is denoted by $\nodes$. For every node $i \in \nodes$, 
the nodal coordinates are given by $\x_i$. 
We denote by $\nnodes=\mathrm{card}(\nodes)$ 
the total number of nodes. The set of nodes 
belonging to a particular element 
$\element \in \mesh$ is denoted by 
$\nodes(\element)\doteq\{i\in\nodes : \x_i\in \element\}$. 
Moreover, $\domain_i$ is the macroelement 
made up of elements that 
contain node $i$, i.e., $\domain_i\doteq 
\bigcup_{\element\in\mesh, \; \x_i\in \element}  
\element$. To simplify the discussion below, 
we use $i$ for both the node and its associated index.

We will make use of linear/bilinear finite elements.
Define $\bfespace = (V_h)^m$, where for simplex meshes,
$V_h \doteq \big\{v_h\in\mathcal{C}^0
(\domain) : v_h|_K\in 
\mathcal{P}_1(\element) \ \forall K\in\mesh\big\}$, 
where $m$ is the number of components of $\contunk$, 
and $\mathcal{P}_1(\domain)$ is the space of 
polynomials of total degree less than or equal 
to one. For $d$-cube partitions, 
$V_h \doteq \big \{v_h\in \mathcal{C}^0(\domain) : 
v_h|_K\in {Q}_1(\element), \forall K\in\mesh\big\}$, 
where $Q_1(\element)$ is space of polynomials 
of partial degree less than or equal to one. 
Any function $\btest\in\bfespace$ is 
a linear combination of the basis 
$\{\shapef[i]\}_{i\in\nodes}$ with nodal 
values $\boldsymbol{v}_i$, where $\shapef[i]$ 
is the shape function associated to the node $i$. Hence, $\btest=\sum_{i\in\nodes}\shapef[i]\boldsymbol{v}_i$.

We use standard notation for Lebesgue spaces. 
The $\ltwo(\omega)$ scalar product is denoted 
by $(\cdot,\cdot)_\omega$ for any set $\omega$. 
However, we omit the subscript for 
$\omega\equiv\domain$. The $\ltwo$ norm is 
denoted by $\Vert\cdot\Vert$. 

The scheme will be based on the method of lines.
First, we present the spatial discretization.
The solution is approximated in space 
using the \ac{fe} spaces defined above, i.e., 
$\contunk\approx\bunk=\sum_{i\in\nodes} 
\shapef[i]\contunk_i$.
In addition, we use the group-FEM technique to 
approximate the fluxes \cite{Fletcher1983,Mabuza2018,Kuzmin2012,Barrenechea2017}. 
Thus, the fluxes are approximated as follows 
$\conv(\contunk_h) \approx \conv_h(\contunk_h) 
:= \sum_{i\in\nodes} \shapef[i]\conv(\contunk_i)$.
To get the semi-discrete scheme, we test the strong form against $\btest\in\bfespace$. Afterwards, we apply integration by parts in the convective term.
Therefore, the semi-discrete problem is given as follows:
find $\bunk\in\bfespace$, with
$\bunk=\boldsymbol{u}_{0h}$ at $t=0$, such that
\begin{equation} \label{discrete-problem}
(\partial_t\bunk,\btest) - 
(\conv_h (\bunk),\gradient\btest) 
+(\normal_{\partial\domain}
\cdot\conv_h (\bunk),\btest)_{\partial\domain} 
= (\bforce,\btest),\; \mbox{for all } \btest\in\bfespace,
\end{equation}
where 
$\contunk_{0h}$ is an LED projection of
$\contunk_0$ onto $\bfespace$ \cite{Kuzmin2012}. 
Note that for the sake of 
simplicity boundary conditions are 
strongly imposed on the conserved variables. 
This procedure implies that the number 
of unknowns in the discrete system is reduced 
after prescribing boundary conditions. Therefore, 
the equations corresponding to prescribed values
are also removed from the algebraic system.
This practice is suitable only for 
simple problems. Transonic 
flow and other complex problems may require 
fairly complex boundary conditions. 
These may include numerical flux boundary 
conditions \cite{Gurris2009}.

For the time discretization, 
we only consider the \ac{be} scheme. 
Other time integrators may be considered,
such as the \ac{ssp} \ac{rk} 
methods (see \cite{Gottlieb2001}). 
We will now discretize the semi-discrete
scheme in time to get fully discrete scheme.
Consider a partition of the time 
domain $(0,T]$ into $n^{ts}$ sub-intervals 
of length $\Delta t_{n+1} = t^{n+1}-t^n$. 
Then, at every time step $n = 0,\ldots,n^{ts}-1$, 
the fully discrete scheme is
\begin{equation}\label{eq.discrete-problem}
\M \derd_t \U^{n+1} + \K(\bunk^\np)\U^{n+1} = \G,
\end{equation}
where $\U^{n+1}\doteq [\contunk_1^{n+1},...,\contunk_{\nnodes}^{n+1}]^T$ is the vector of nodal values at time $t^{n+1}$, and $\derd_t (\U) \doteq \dtnp^{-1} (\U^{n+1} - \U^n)$. 
$\M$ and $\K$ are block matrices. $\G$ is a block vector. They are
given by
\begin{align}
\M &= \{ \M_{ij} \}_{i,j=1}^N,\\
\K &= \{ \K_{ij} \}_{i,j=1}^N,\\
\G &= \{ \G_{i} \}_{i=1}^N.
\end{align}
Each $\M_{ij}$ is an $m\times m$ matrix. So is $\K_{ij}$.
$\G_{i}$ is an $m\times 1$ vector.
These block matrix and vector entries are given for 
any $i,j\in\nodes$ by\\

\begin{tabular}{ll}
$\M_\ij = \{ \M_\ij^{\beta\gamma} \}_{\beta,\gamma=1}^m,$
&
$\M_\ij^{\beta\gamma} \doteq(\shapef[j],\shapef[i]) \delta_{\beta\gamma},$\\
$\K_\ij = \{ \K_\ij^{\beta\gamma} \}_{\beta,\gamma=1}^m,$
&
$\K_\ij^{\beta\gamma} \doteq -(\shapef[j]\delta_{\beta\xi},\conv^{\prime}_k (\contunk_j^\np)^{\xi\eta}\cdot\partial_k \shapef[i]\delta_{\eta\gamma}) + (\shapef[j]\delta_{\beta\xi},n_k\cdot\conv^\prime_k(\contunk_j^\np)^{\xi\eta} \shapef[i]\delta_{\eta\gamma})_{\partial\domain},$\\
$\G_i = \{ \G_i^\beta \}_{\beta=1}^m,$ 
&
$\G_i^\beta \doteq(\bforce^\beta,\shapef[i]),$
\end{tabular}\\
where Einstein summation applies over $k$, $\xi$ and $\eta$. $\beta,\gamma,\xi,\eta \in  \{1,\ldots,m\}$ are the component indices, and $\delta_{\beta\gamma}$ is the Kronecker delta. Notice that we have rearranged the terms in $\K_\ij^{\beta\gamma}$ using the fact that the Euler flux is homogeneous of degree one, i.e., $\conv (\contunk_i) = \conv' (\contunk_i)\contunk_i$.

\subsection{Stabilization properties}\label{sec.stability}
In this section, we introduce some concepts 
required for discussing the stabilization 
presented in section \ref{sec.stabilization}. 
Stabilization for hyperbolic systems is developed
from techniques designed for scalar equations. 
We briefly review below the ideas that govern
stabilization for scalar problems.

\begin{definition}[Local Discrete Extremum]\label{def.extremum}
The function $\test\in\fespace$ has a local 
discrete minimum (resp. maximum) on 
$i\in\nodes$ if $u_i\leq u_j$ (resp. $u_i\geq u_j$) 
${\forall j\in\neighborhood[i]}$.
\end{definition}
\begin{definition}[Local \ac{dmp}]\label{def.local-dmp}
A solution $\unk\in\fespace$ satisfies the local 
discrete maximum principle if for every $i\in\nodes$
\begin{equation}
\min_{j\in\neighborhood[i]\backslash\{i\}} u_j \leq u_i \leq 
\max_{j\in\neighborhood[i]\backslash\{i\}} u_j.
\end{equation}
\end{definition}
\begin{definition}[LED]\label{def.led}
A scheme is local extremum diminishing if, for 
every $u_i$ that is a local discrete maximum (resp. minimum),
\begin{equation}
\frac{\der u_i}{\der t} \leq 0, \qquad \qquad 
\left(resp.\ \frac{\der u_i}{\der t} \geq 0 \right),
\end{equation}
is satisfied.
\end{definition}

Given a scalar problem with the forcing term $\bm g \equiv 0$,
a semi-discrete scheme 
may be written 
as 
\begin{equation}
\sum_j (\M_\ij \frac{du_j}{dt} + \A_\ij u_j) = 0,
\end{equation}
plus appropriate boundary conditions.
One possible strategy for satisfying the above 
properties consists of designing a scheme 
that yields a positive diagonal mass matrix $\M_{ij} = m_i\delta_\ij$,
where $m_i > 0$, and a matrix $\A$ that satisfies 
\begin{equation}\label{eq.m-matrix}
\sum_j \A_\ij = 0, \quad\text{and}\quad \A_\ij \leq  0, \ i\neq j. 
\end{equation}
In this case, it is possible to rewrite the system as
\begin{equation}\label{eq.general-led-problem}
m_i \frac{d u_i}{d t}+ \sum_{j\in\neighborhood[i]\backslash\{i\}} \A_\ij (u_j - u_i) = 0, \quad \forall \, i\in\nodes.
\end{equation}
As shown in \cite{codina_discontinuity-capturing_1993} and \cite{kuzmin_flux_2002}, such a scheme satisfies the local \ac{dmp} for steady problems and it is also \ac{led} when applied to transient problems. 

The extension of these properties to hyperbolic systems is based on analyzing them in characteristic variables.
Let us consider a one-dimensional linear hyperbolic system with a constant Jacobian flux, $\conv^\prime$. In this particular case, the continuous system can be diagonalized. Thus it is possible to discretize and solve for the characteristic variables. For example, for  the set of characteristics variables, say $W$, the continuous system reads: 
\begin{equation}\label{eq.diagonal-system}
\partial_t W + \Lambda \partial_x W = 0,
\end{equation}
where $\Lambda=diag(\lambda_1,...,\lambda_m)$ is a diagonal $m$ by $m$ matrix.
At this point, one can see the system as a set of independent scalar transport problems.
Thus, it leads to a system with diagonal blocks after discretizing it with \acp{fe}. 

Assuming conditions \eqref{eq.m-matrix} are satisfied for every component of problem \eqref{eq.diagonal-system}, then the scheme will be \ac{led} for each characteristic variable. 
Notice that this is equivalent to forcing the original (coupled) \ac{fe} approximation to have negative semi-definite off-diagonal blocks. 
That is, the \ac{fe} discretization of the problem in characteristic variables reads
\begin{equation}\label{eq.fe-diagonal-system}
(\shapef[j],\shapef[i]) \partial_t W_j + \Lambda (\partial_x\shapef[j],\shapef[i])W_j = 0.
\end{equation}
Since in this case it is a one dimensional linear problem, we can recover the original problem using the fact that $W=R^{-1}U$, and $\conv^\prime = R\Lambda R^{-1}$. Multiplying \eqref{eq.fe-diagonal-system} at the left by $R$,
\begin{equation}
(\shapef[j],\shapef[i]) \partial_t R R^{-1}U_j + R\Lambda (\partial_x\shapef[j],\shapef[i])R^{-1} U_j = 0.
\end{equation}
In this case, $(\partial_x\shapef[j],\shapef[i])$ is simply a scalar value. Hence, we are able to recover the original (coupled) problem \ac{fe} discretization.
\begin{equation}
(\shapef[j],\shapef[i]) \partial_t U_j + \conv^\prime (\partial_x\shapef[j],\shapef[i]) U_j = 0.
\end{equation}
Thus, if $\conv^\prime (\partial_x\shapef[j],\shapef[i])$ is negative semi-definite for $j\neq i$, then the problem in characteristic variables will satisfy conditions \eqref{eq.m-matrix} for each variable.

In the case of more general multidimensional 
problems (e.g. Euler equations), this would 
\emph{only} imply that the scheme is \ac{led} 
for a certain set of \emph{local} characteristic 
variables. Furthermore, if the flux Jacobian 
$\conv^\prime$ is not linear, then even the 
definition of the matrix $\A_\ij$ 
(relating nodes $i$ and $j$) is not trivial. 
Let us recall the definition of these blocks 
for Euler equations
\begin{equation}
\K_\ij \doteq -(\shapef[j],\conv^{\prime}_k (\contunk_j) \cdot\partial_k \shapef[i]) + (\shapef[j],n_k\cdot\conv^\prime_k(\contunk_j) \shapef[i])_{\partial\Omega} = \conv^{\prime}_k (\contunk_j) \cdot(\partial_k\shapef[j], \shapef[i]) ,
\end{equation}
where we have undone integration by parts. It is easy to check that $\sum_j (\partial_k\shapef[j], \shapef[i]) = 0$. Hence, we can write
\begin{equation}
\sum_{j\in\neighborhood[i]} \K_\ij \contunk_j = 
\sum_{j\in\neighborhood[i]\backslash\{i\}} (\partial_k\shapef[j], \shapef[i]) (\conv^{\prime}_k (\contunk_j) \cdot \contunk_j - \conv^{\prime}_k (\contunk_i) \cdot\contunk_i).
\end{equation}
As previously stated, it is not straightforward in the case of Euler equations to rewrite the discrete problem in the form of \eqref{eq.general-led-problem}. However, making use of special density-averaged variables it is possible to rewrite the previous expression as
\begin{equation}
\sum_{j\in\neighborhood[i]} \K_\ij \contunk_j = 
\sum_{j\in\neighborhood[i]\backslash\{i\}} \conv^{\prime}_k (\contunk_\ij)\cdot (\partial_k\shapef[j], \shapef[i]) (\contunk_j - \contunk_i),
\end{equation}
where $\contunk_\ij$ are the Roe mean values \cite{Roe1981}. For an ideal gas, these are defined as
\begin{equation}
\rho_\ij = \sqrt{\rho_i\rho_j},\quad
\mm_\ij = \frac{\mm_i\sqrt{\rho_j}+\mm_j\sqrt{\rho_i}}{\sqrt{\rho_i}+\sqrt{\rho_j}}, \quad
(\rho E)_\ij = \frac{1}{\gamma} \left(\rho_\ij H_\ij + (\gamma-1)\frac{|\mm_\ij|^2}{2\rho_\ij}\right),
\end{equation}
where $H_\ij$ is the average enthalpy
\begin{equation}
H_\ij = \frac{H_i\sqrt{\rho_i}+H_j\sqrt{\rho_j}}{\sqrt{\rho_i}+\sqrt{\rho_j}},\quad\text{and}\quad H_i= \frac{ {\rho E}_i+p_i}{\rho_i}.
\end{equation}
Therefore, using this density-averaged variables it is possible to rewrite Euler problem in the form of \eqref{eq.general-led-problem}. Hence, if $- \conv^{\prime}_k (\contunk_\ij)\cdot (\partial_k\shapef[j], \shapef[i])$ has non-positive eigenvalues, then the scheme will be \ac{led} for a certain set of \emph{local} characteristic variables.
Schemes that satisfy this property are named \emph{local bounds preserving schemes} in the literature \cite{Mabuza2018}. 
This reasoning above motivated the definition of the \ac{led} \emph{principle} for hyperbolic systems of equations by Kuzmin \cite{Kuzmin2005} and coworkers.
Adapted from this principle, we define local bounds preserving schemes as follows.
\begin{definition}\label{def.system-led}
	The semi-discrete scheme
	\begin{equation}
	m_i\partial_t \contunk_i + \sum_{j\neq i} \A_\ij (\contunk_j - \contunk_i ) = \boldsymbol{0}
	\end{equation}
	is said to be local bounds preserving if $\M$ is diagonal with positive entries (i.e. $\M_\ij=m_i\delta_\ij I_{m\times m}$), $\A_\ij$ has non-positive eigenvalues for every $j\neq i$, and $\sum_j\A_\ij=\boldsymbol{0}$.
\end{definition}
Unfortunately, to the best of our knowledge, 
satisfying this definition does not guarantee 
positivity for density or internal energy. It also
does not guarantee non-decreasing entropy. In any case, 
numerical schemes based on this definition 
have shown good numerical behavior \cite{Kuzmin2003,Kuzmin2005,Lohmann2016,Mabuza2018}. 
In the stabilization presented in these papers,
conservative artificial diffusion is used.
There are different types of artificial diffusion.
The simplest is scalar artificial diffusion which is based 
on the spectral radius of 
$\A_\ij$ \cite{Lohner2004,Kuzmin2005}. 
This diffusion is also called Rusanov artificial diffusion, 
since for linear \acp{fe} or finite volume 
methods in one dimension, 
the scheme is equivalent to the Rusanov approximate 
Riemann solver \cite{Kuzmin2005,Toro2009}.
Without any special treatment, adding artificial
diffusion results in a first order accurate 
scheme. The key to recovering high-order 
convergence is to modulate the action of 
the artificial diffusion term, and restrict 
its action to the vicinity of discontinuities. 
Here, we construct a stabilization term using 
Rusanov artificial diffusion and a 
differentiable shock detector recently 
developed for scalar problems \cite{badia_monotonicity-preserving_2017,Bonilla2019a}.
This will be discussed extensively in the next section.

\section{Nonlinear stabilization}\label{sec.stabilization}

In the previous section, we mentioned that 
the Galerkin \ac{fe} discretization yields 
oscillatory solutions in regions around 
discontinuities. We supplement the original 
scheme with an artificial diffusion term to 
stabilize it and mitigate these oscillations. 
The proposed stabilization term is given by
\begin{equation}\label{eq.stabilization}
B_h(\boldsymbol{w}_h;\bunk,\btest)\doteq \sum_{\element[e]\in\mesh}\sum_{i,j \in \nodes(\element[e])} \artdif^e_\ij(\boldsymbol{w}_h)\graphl(i,j) \boldsymbol{v}_i \cdot  I_{m\times m} \boldsymbol{u}_j,
\end{equation}
for any $\bunk\in\bfespace$ and $\btest\in\bfespace$. Here, $\graphl(i,j)\doteq2\delta_\ij-1$ is a graph Laplacian operator defined in \cite{badia_monotonicity-preserving_2017}, and $\artdif^e_\ij(\boldsymbol{w}_h)$ is the element-wise artificial diffusion defined as
\begin{equation}\label{eq.artificial-diffusion}
\begin{aligned}
\artdif^e_\ij(\boldsymbol{w}_h) &  \doteq\max\left(\vdetector[i](\boldsymbol{w}_h)\lambdamax[ij],\vdetector[j](\boldsymbol{w}_h)\lambdamax[ji]\right), \quad  
\text{for } j\in\neighborhood[i]\backslash\{i\}, \\
\artdif^e_{ii}(\boldsymbol{w}_h)&\doteq\displaystyle \sum_{j\in\neighborhood[i]\backslash\{i\}} \artdif^e_\ij(\boldsymbol{w}_h),\quad
\end{aligned}
\end{equation}
where $\lambdamax[ij]$ is the spectral radius of the elemental convection matrix relating nodes $i,j\in\nodes$. We will also denote the spectral radius by $\sr(\cdot)$, i.e., $\sr\left(\conv^\prime(\contunk_\ij)\cdot(\gradient\shapef[j],\shapef[i])_{\element[e]}\right)$. As previously introduced, this artificial diffusion term is based on Rusanov scalar diffusion \cite{Kuzmin2012}. It is important to mention that the eigenvalues of these matrices can be easily computed as
\begin{equation}\label{eq.eigenvalues}
\lambda_{1,..,d}=\vel_\ij \cdot \cc_\ij^e,\quad \lambda_{d+1}=\vel_\ij\cdot \cc_\ij^e - c\|\cc_\ij^e\|, \quad \lambda_{d+2}=\vel_\ij\cdot \cc_\ij^e + c\|\cc_\ij^e\|
\end{equation}
where 
\begin{equation}
\cc_\ij^e = (\gradient\shapef[j],\shapef[i])_{\element[e]}, \quad \text{and}\quad c = \sqrt{(\gamma -1)\left(H_\ij - \frac{|\mm_\ij|^2}{2\rho_\ij^2}\right)}.
\end{equation}

We denote by $\vdetector[i](\boldsymbol{w}_h)$ the shock detector used for modulating the action of the artificial diffusion term. The idea behind the definition of this detector is minimizing the amount of artificial diffusion introduced while stabilizing any oscillatory behavior. In regions where the local \ac{dmp} (see Def. \ref{def.local-dmp}) is not satisfied for any chosen set of components, we ensure that Def. \ref{def.system-led} is satisfied. $\vdetector[i](\boldsymbol{w}_h)$ must be a positive real number which takes value 1 when $\unk(\x_i)$ is an inadmissible value of $\bunk$, and smaller than 1 otherwise. To this end, we define 
\begin{equation}\label{eq.system-detector}
\vdetector[i](\bunk)\doteq \textstyle\max \{\detector[i](\unk^\beta)\}_{\beta\in C},
\end{equation}
where $C$ is the set of components that are used to detect inadmissible values of $\bunk$, e.g. density and total energy in the case of Euler equations. For simplicity, we restrict ourselves to the components of $\bunk$. However, derived quantities such as the pressure or internal energy can be also used.

In order to introduce the shock detector, let us recall some useful notation from \cite{badia_monotonicity-preserving_2017}. Let $\rr_{\ij} = \x_j - \x_i$ be the vector pointing from node $\x_i$ to $\x_j$ with $i,j\in\nodes$ and $\hat{\rr}_{\ij} \doteq \frac{\rr_{\ij}}{|\rr_{\ij}|}$. Recall that the set of points $\x_j$ for $j\in\neighborhood[i]\backslash\{i\}$ define the macroelement $\domain_i$ around node $\x_i$. Let $\x_\ij^\sym$ be the point at the intersection between $\partial\domain_i$ and the line that passes through $\x_i$ and $\x_j$ that is not $\x_j$ (see Fig. \ref{fig.usym}). The set of all $\x_\ij^\sym$ for all $j\in\neighborhood[i]\backslash\{i\}$ is represented with $\symneigh[i]$. We define $\rr_\ij^\sym \doteq \x_\ij^\sym - \x_i$. Given $\x^\sym_\ij$ in two dimensions, let us call $a$ and $b$ the indices of the vertices such that they define the edge in $\partial\domain_\i$ that contains $\x_\ij^\sym$. We define $\contunk_j^\sym$ as the value of $\bunk$ at $\x_\ij^\sym$, i.e. $\bunk(\x_\ij^\sym)$.
\begin{figure}[h]
\centering
\includegraphics[width=0.27\textwidth]{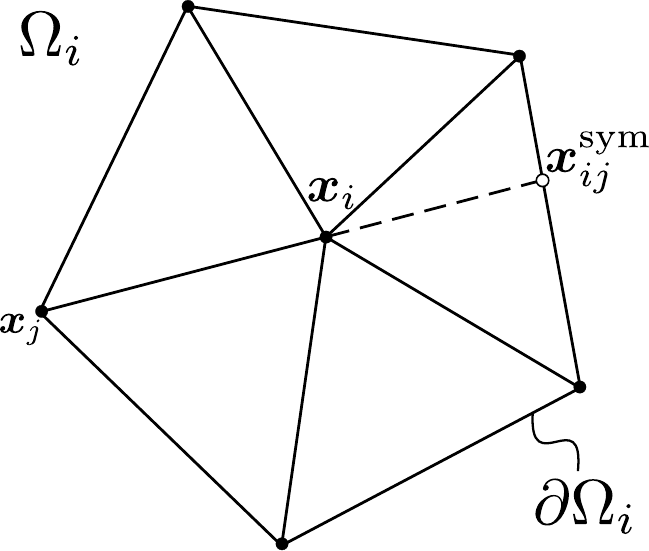}
\caption{$u^\sym$ drawing}
\label{fig.usym}
\end{figure}

Both $\contunk_\ij^\sym$ and $\x_\ij^\sym$ are only required to construct a linearity preserving shock detector. Let us define the jump and the mean of a linear approximation of component $\beta$ of the unknown gradient at node $\x_i$ in direction $\rr_\ij$ as
\begin{equation}\label{eq.jump}
\jump{\gradient \unk^\beta}_\ij \doteq \frac{u_j^\beta - u_i^\beta}{|\rr_\ij|} + \frac{u_j^{\sym,\beta} - u^\beta_i}{|\rr_\ij^\sym|}, 
\end{equation}
\begin{equation}\label{eq.mean}
\mean{|\gradient \unk^\beta\cdot \hat{\rr}_\ij|}_{\ij} \doteq \frac{1}{2}
\left(\frac{|u_j^\beta-u_i^\beta|}{|\rr_\ij|}+\frac{|u_j^{\sym,\beta}-u_i^\beta|}{|\rr_\ij^\sym|}\right).
\end{equation}
In the present work, for each component in $C$, we use the same shock detector developed in \cite{badia_monotonicity-preserving_2017}. Let us recall its definition
\begin{equation}\label{eq.detector}
\detector[i](\unk^\beta)\doteq\left\{\begin{array}{cc}
\left[\dfrac{\left|\sum_{j\in\neighborhood[i]} \jump{\gradient\unk^\beta}_\ij\right|}{\sum_{j\in\neighborhood[i]}2\mean{\left|\gradient\unk^\beta\cdot\hat{\rr}_\ij\right|}_\ij}\right]^q & \text{if}\quad \sum_{j\in\neighborhood[i]} \mean{\left|\gradient\unk^\beta\cdot\hat{\rr}_\ij\right|}_\ij \neq 0 \\
0 & \text{otherwise}
\end{array}
\right. .
\end{equation}
From \cite[Lm. 3.1]{badia_monotonicity-preserving_2017} we know that \eqref{eq.detector} is valued between 0 and 1, and it is only equal to one if $\unk^\beta(\x_i)$ is a local discrete extremum (in a space-time sense as in Def. \ref{def.extremum}).
Since the linear approximations of the unknown gradients are exact for $\unk^\beta\in\mathcal{P}_1$, the shock detector vanishes when the solution is linear. Thus, it is also linearly preserving for every component in $C$. This result follows directly from \cite[Th. 4.5]{badia_monotonicity-preserving_2017}.

The final stabilized problem in matrix form reads as follows. Find $\bunk\in\bfespace$ such that $\unk^\beta=\bar{u}^\beta_h$ on $\inflowboundary^\beta$, $\bunk=\boldsymbol{u}_{0h}$ at $t=0$, and
\begin{equation}\label{eq.stabilized-matrix-problem}
\oM(\bunk^\np)\derd_t\U^\np+\oK(\bunk^\np)\U^\np = \G
\end{equation}
for $n=1,...,n^{ts}$, where 
\begin{equation}
 \oM_\ij(\bunk^\np)\doteq \left[1-\max\left(\vdetector[i],\vdetector[j]\right)\right](\shapef[j],\shapef[i])I_{m\times m} + \max\left(\vdetector[i],\vdetector[j]\right)(\delta_\ij,\shapef[i])I_{m\times m},
\end{equation}
$\oK_\ij(\bunk^\np) \doteq \K_\ij + \B_\ij$, and $\B_\ij(\bunk)\doteq \sum_{\element[e]\in\mesh}\artdif^e_{\ij}(\bunk)\ \graphl(i,j)  I_{m\times m}$.

\begin{lemma}[Local bounds preservation]\label{thm-discreteLED}
	Consider $\bunk\in\bfespace$ with component $\beta$ in the set of tracked variables $C$. The stabilized problem \eqref{eq.stabilized-matrix-problem}, with $\G= \0$, is local bounds preserving as defined in Def. \ref{def.system-led} at any region where $\unk^\beta$ has extreme values.
\end{lemma}
\begin{proof}
	If component $\beta\in C$ of $\bunk$ has an extremum at $\x_i$, then from \cite[Lm. 3.1]{badia_monotonicity-preserving_2017} we know that $\detector[i](\unk^\beta)=1$. Moreover, from \eqref{eq.system-detector} is easy to see that $\vdetector[i](\bunk)=1$. In this case, $\oM_\ij(\unk) = (\delta_\ij,\shapef[i]) I_{m\times m}$. Hence, $\oM_\ij(\unk) = 0$ for $j\neq i$ and $\oM_{ii}(\unk) = m_i$. Therefore, we can rewrite the system as follows
	\begin{align}
	m_i\partial_t \contunk_i &+ \sum_{j\in\neighborhood[i]\backslash\{i\}} \oK_\ij(\contunk_\ij) (\contunk_j - \contunk_i) = \\
	m_i\partial_t \contunk_i &+ \sum_{j\in\neighborhood[i]\backslash\{i\}}\sum_{\element[e]\in\mesh} \left(\conv^\prime(\contunk_\ij)\cdot(\gradient\shapef[j],\shapef[i])_{\element[e]}-\artdif^e_\ij I_{m\times m}\right) (\contunk_j - \contunk_i) = \0.
	\end{align}
	We need to prove that the eigenvalues of $\oK_\ij(\contunk_\ij)$ are non-positive. To this end, let us show the following inequality holds
	\begin{equation}
	\sum_{\element[e]\in\mesh}\sr\left(\conv^\prime(\contunk_\ij)\cdot(\gradient\shapef[j],\shapef[i])_{\element[e]}\right)\geq\sr(\conv^\prime(\contunk_\ij)\cdot(\gradient\shapef[j],\shapef[i])).
	\end{equation}
	From \eqref{eq.eigenvalues}, it is easy to check that $\sr\left(\conv^\prime(\contunk_\ij)\cdot(\gradient\shapef[j],\shapef[i])_{\element[e]}\right) = |\vel_\ij\cdot\cc_\ij^e| + c\|\cc_\ij^e\|$. Since  $\cc_\ij=\sum_{\element[e]\in\mesh}\cc_\ij^e$, we have that
	\begin{equation}
	\sum_{\element[e]\in\mesh}	\left|\vel_\ij\cdot\cc_\ij^e\right| \geq \left|\vel_\ij\cdot\cc_\ij\right|, \quad \text{and}\quad 
	\sum_{\element[e]\in\mesh}	c\|\cc_\ij^e\| \geq c\|\cc_\ij\|.
	\end{equation}
	Therefore, $\sum_e\sr(\K_\ij^e(\contunk_\ij))\geq\sr(\K_\ij(\contunk_\ij))$. 
	Moreover, by definition (see \eqref{eq.artificial-diffusion}),
	\begin{equation}
	\artdif_\ij^e\geq\sr\left(\conv^\prime(\contunk_\ij)\cdot(\gradient\shapef[j],\shapef[i])_{\element[e]}\right) \quad \text{for} \quad j\neq i.
	\end{equation}
	Furthermore, from \eqref{eq.stabilization}, is easy to see that $\sr(\B_\ij^e(\contunk_\ij))\geq\sr(\K_\ij^e(\contunk_\ij))$, hence $\sr(\B_\ij(\contunk_\ij))\geq\sr(\K_\ij(\contunk_\ij))$. Finally, 
	since $\oK_\ij= \K_\ij + \B_\ij$ and $\B_\ij = \sum_e\B^e_\ij = \sum_e-\artdif_\ij^e I_{m\times m}$ for all $j\neq i$, then the maximum eigenvalue of $\oK_\ij(\contunk_\ij)$ is non-positive, which completes the proof.
\end{proof}

\subsection{Differentiability}\label{sec.differentiability}
In the case of steady, or implicit time integration, differentiability plays a role in the convergence behavior of the nonlinear solver. This is specially important if one wants to use Newton's method. In the case of scalar problems it has been shown in \cite{badia_differentiable_2017,badia_monotonicity-preserving_2017} that convergence is greatly improved after few modifications to make a scheme twice-differentiable. In this section, we introduce a set of regularizations applied to all non-differentiable functions present in the stabilized scheme introduced above. In order to regularize these functions, we follow a similar strategy as \cite{badia_differentiable_2017,badia_monotonicity-preserving_2017}. Absolute values are substituted by
\begin{equation}
\absn{x} =  \sqrt{x^2 + \varepsilon_h},\qquad
\absd{x} =  \frac{x^2}{\sqrt{x^2 + \varepsilon_h}}.
\end{equation}
Note that $\absd{x}\leq |x|\leq\absn{x}$. Next, we also use a smooth maximum function, $\smax(\cdot)$, as
\begin{equation}\label{eq:smax}
\smax(x,y) \doteq \frac{ \absn[\sigma_h]{x -y}}{2} +\frac{x +y}{2} \geq \max(x,y).
\end{equation}
In addition, we need a smooth function to limit the value of any given quantity to one. To this end, we use
\begin{equation}\label{eq.smthlimit}
\smthlimit{x} \doteq \left\{\begin{array}{ll}
2x^4-5x^3+3x^2+x, & x<1,\\
1, & x\geq 1.
\end{array}\right.
\end{equation}

The set of twice-differentiable functions defined above allows us to redefine the stabilization term introduced in Sect. \ref{sec.stabilization}. In particular, we define 
\begin{equation}
\tilde{B}_h(\boldsymbol{w}_h;\bunk,\btest)\doteq \sum_{\element[e]\in\mesh}\sum_{i,j \in \nodes(\element[e])} \smthartdif^e_\ij(\boldsymbol{w}_h)\graphl(i,j) \boldsymbol{v}_i \cdot  I_{m\times m} \boldsymbol{u}_j,
\end{equation}
where
\begin{equation}\label{eq.smth-artificial-diffusion}
\begin{aligned}
\smthartdif^e_\ij(\boldsymbol{w}_h) &  \doteq\smax\left(\vsmthdetector[i](\boldsymbol{w}_h)\lambdamax[ij],\vsmthdetector[j](\boldsymbol{w}_h)\lambdamax[ji]\right), \quad  
\text{for } j\in\neighborhood[i]\backslash\{i\}, \\
\smthartdif^e_{ii}(\boldsymbol{w}_h)&\doteq\displaystyle \sum_{j\in\neighborhood[i]\backslash\{i\}} \smthartdif^e_\ij(\boldsymbol{w}_h).\quad
\end{aligned}
\end{equation}
Let us note that $\lambdamax[\ij]$ needs to be regularized as $\lambdamax[\ij]=\absn{\vel_\ij \cc_\ij^e} + c\|\cc_\ij^e\|$. 
The shock detector is also redefined to use the regularized version of the shock detector, which reads
\begin{equation}\label{eq.smth-system-detector}
\vsmthdetector[i](\bunk)\doteq \textstyle\smax \{\smthdetector[i](\unk^\beta)\}_{\beta\in C}.
\end{equation}
In the case of the component shock detector we recall the definition in \cite[eq. 18]{badia_monotonicity-preserving_2017}
\begin{equation}\label{eq.smth-detector}
\smthdetector[i](\unk^\beta)\doteq
\left[\smthlimit{\dfrac{\absn{\sum_{j\in\neighborhood[i]} \jump{\gradient\unk^\beta}_\ij }+\zeta_h}{\sum_{j\in\neighborhood[i]}2\mean{\absd{\gradient\unk^\beta\cdot\hat{\rr}_\ij}}_\ij+\zeta_h}}\right]^q,
\end{equation}
where $\zeta_h$ is a small value for preventing division by zero. Finally, the twice-differentiable stabilized scheme reads:\\
Find $\bunk\in\bfespace$ such that $\unk^\beta=\bar{u}^\beta_h$ on $\inflowboundary^\beta$, $\bunk=\boldsymbol{u}_{0h}$ at $t=0$, and
\begin{equation}\label{eq.smth-stabilized-matrix-problem}
\sM(\bunk^\np)\derd_t\U^\np+\sK(\bunk^\np)\U^\np = \G \quad\text{for}\ n=1,...,n^{ts},
\end{equation}
where 
\begin{align}
\sM_\ij(\bunk^\np) & \doteq \left[1-\smax\left(\vsmthdetector[i],\vsmthdetector[j]\right)\right](\shapef[j],\shapef[i])I_{m\times m} + \smax\left(\vsmthdetector[i],\vsmthdetector[j]\right)(\delta_\ij,\shapef[i])I_{m\times m},\\
\sK_\ij(\bunk^\np) & \doteq \K_\ij(\bunk^\np) + \tilde{\B}_\ij(\bunk^\np),
\end{align}
and $\tilde\B_\ij(\bunk)\doteq \sum_{\element[e]\in\mesh}\smthartdif^e_{\ij}(\bunk)\ \graphl(i,j)  I_{m\times m}$, for $i,j\in\nodes$.

\begin{corollary}
	The differentiable scheme in Eq. \eqref{eq.smth-artificial-diffusion} is local bounds preserving, as defined in Def. \ref{def.system-led}, at any region where $\unk^\beta$ has extreme values for every $\beta$ in $C$. 
\end{corollary}
\begin{proof}	
	For an extreme value of $\unk^\beta$, since $\absd{x}\leq |x|\leq\absn{x}$ the quotient of \eqref{eq.smth-detector} is larger than one. Hence, by definition of $Z(x)$,  $\smthdetector[i]$ is equal to 1. At this point, it is easy to check that $\smthartdif_\ij^e\geq\artdif_\ij^e$ in virtue of the definition of $\smax$. Therefore, $\sr(\tilde{\B}^e_\ij(\bunk)) \geq \sr(\B^e_\ij(\bunk))$, completing the proof.
\end{proof}

Moreover, it is important to mention that the differentiable shock detector is weakly linearly-preserving as $\zeta_h$ tends to zero. This result follows directly from \cite{badia_monotonicity-preserving_2017}. In order to obtain a differentiable operator, we have added a set of regularizations that rely on different parameters, e.g., $\sigma_h,\ \varepsilon_h,\ \zeta_h$. Giving a proper scaling of these parameters is essential to recover theoretic convergence rates. In particular, we use the following relations
\begin{equation}\label{eq.scaling}
\sigma_h=\sigma|\lambdamax|^2  L^{2(d-3)} h^{4},\quad \varepsilon_h=\varepsilon L^{-4} h^2,\quad \zeta_h=L^{-1}\zeta,
\end{equation}
where $d$ is the spatial dimension of the problem, $L$ is a characteristic length, and $\sigma,\ \varepsilon,$ and $\zeta$ are of the order of the unknown.

\section{Nonlinear solver}\label{sec.solver}
In this section, we describe the method used for solving the nonlinear system of equations arising from the scheme introduced above. In particular, we use a hybrid Picard--Newton approach in order to increase the robustness of the nonlinear solver. Moreover, for the differentiable version we also use a continuation method to improve the nonlinear convergence.

We represent the residual of the equation \eqref{eq.smth-stabilized-matrix-problem} at the $k$-th iteration by $\R(\bunk^\nk)$, i.e., 
\begin{equation}\label{eq.residual}
\R(\bunk^\nk) \doteq \sM(\bunk^\nk)\derd_t\U^\nk+\sK(\bunk^\nk)\U^\nk - \G.
\end{equation}
Hence, the Jacobian is defined as
\begin{align}\label{eq.jacobian}
\J(\bunk^\nk)&\doteq \frac{\partial\R(\bunk^\nk)}{\partial \U^\nk} \\
&= \dtnp^{-1}\sM(\bunk^\nk)+\sK(\bunk^\nk) + \dtnp^{-1}\frac{\partial\sM(\bunk^\nk)}{\partial \U^\nk}\derd_t\U^\nk+\frac{\partial \sK(\bunk^\nk)}{\partial \U^\nk}\U^\nk.
\end{align}
Therefore, Newton method consists of solving $\J(\bunk^\nk) \Delta\U^{k+1,n+1} = - \R(\bunk^\nk)$. However, it is well known that Newton method can diverge if the initial guess of the solution $\bunk^{0,n+1}$ is not close enough to the solution. In order to improve the robustness, we introduce the following modifications. We use a line-search method to update the solution at every time step. Thus, the new approximation is computed as $\U^{k+1,n+1} = \U^\nk + \lambda\Delta\U^{k+1,n+1}$, where $\lambda$ is computed (approximately) such that it minimizes $\|\R(\bunk^{k+1,n+1})\|$. To approximate $\lambda$ we use a standard golden section search algorithm \cite{Brent1972}. However, any other minimization or backtracking strategy could potentially be used.

As introduced at the beginning of the section, we also use a hybrid approach combining Newton method with Picard linearization. Picard nonlinear iterator can be obtained removing the last two terms of \eqref{eq.jacobian}, i.e.,
\begin{equation}\label{eq.picard}
\left(\dtnp^{-1}\sM(\bunk^\nk)+\sK(\bunk^\nk)\right)\Delta\U^{k+1,n+1} = -\R(\bunk^\nk).
\end{equation}
Clearly, it is equivalent to 
\begin{equation}
\left(\dtnp^{-1}\sM(\bunk^\nk)+\sK(\bunk^\nk)\right)\U^{k+1,n+1} = \sM(\bunk^\nk)\U^{n} + \G.
\end{equation}

Moreover, we modify the definition of left hand side terms in \eqref{eq.picard} to enhance the robustness of the method. In particular, we use $\vdetector[i]=1$ for computing these terms while we use the value obtained from \eqref{eq.system-detector} for the residual. Using this strategy the solution remains unaltered, but the obtained approximations $\bunk^\nk$ for intermediate values of $k$ are more diffusive. Even though this modification slows the nonlinear convergence, it is essential at the initial iterations. Otherwise, the robustness of the method might be jeopardized.

The resulting iterative nonlinear solver consists in the following. We iterate using Picard method in \eqref{eq.picard}, with the modification described above, until the $\ltwo$ norm of the residual is smaller than a given tolerance. In the present work, we use tolerances close to $10^{-2}$. Afterwards, Newton method with the exact Jacobian in \eqref{eq.jacobian} is used until the desired nonlinear convergence criteria are satisfied.

For the differentiable stabilization, we also equip the above scheme with a continuation method on the regularization parameters. In order to accelerate the convergence of the method, we use high values for the parameters during the first iterations. This results in a more diffusive solution, but nonlinear convergence is accelerated. As the nonlinear approximation is closer to the solution, we diminish the value of the parameters to avoid introducing excessive artificial diffusion to the system. This process is preformed gradually as a function of the residual in \eqref{eq.residual}. In particular, we use the following relation
\begin{equation}
\varepsilon^k = \tilde{\varepsilon} \ \frac{\|\R(\bunk^\nk)\|}{\|\R(\bunk^{0,n+1})\|},
\end{equation}
where $\varepsilon^k$ is the effective parameter used in relations \refeq{eq.scaling}, and $\tilde{\varepsilon}$ is parameter defined by the user. We summarize the nonlinear solver introduced above in Alg.\ \ref{alg.scheme}.

\begin{algorithm}[h]
\caption{Hybrid Picard--Newton scheme with the continuation method.}\label{alg.scheme}
\KwIn{$\U^{0,n+1}$, ${\rm tol_1}$, ${\rm tol_2}$, $\varepsilon$, Continuation}
\KwOut{$\U^{k,n+1}$, $k$}
$k=1$, $\varepsilon^1 = \varepsilon$ \\
\While{$ \|\R(\U^\nk)\|/ \|\R(\U^{0,n+1})\| \geq \rm tol_1$ }{
	Compute $\vdetector[i](\U^\nk)$ using \eqref{eq.system-detector} \\
	Compute $\Delta \U^{k+1,n+1}$ using \eqref{eq.picard} \\
	Minimize $\|\R(\U^{k+1,n+1})\|$, where $\U^{k+1,n+1} = \lambda\Delta \U^{k+1,n+1} + \U^\nk$, with respect to $\lambda$ \\
	Set $\U^{k+1,n+1}=\lambda\Delta \U^{k+1,n+1} + \U^\nk$ \\
	\eIf{\rm Continuation }{
		Set $\varepsilon^k = \tilde{\varepsilon}\frac{\|\R(\U^{k+1,n+1})\|}{\|\R(\U^{0,n+1})\|}$ \\
	}{
		Set $\varepsilon^k = \varepsilon$
	}
	Set $\sigma^k = 10^2\, \varepsilon^k $ \\
	Update $k= k+1$
}

\While{$ \|\R(\U^\nk)\|/ \|\R(\U^{0,n+1})\| \geq \rm tol_2$ }{
	Compute $\vdetector[i](\U^\nk)$ using \eqref{eq.system-detector} \\
	Solve $\J(\U^\nk)\Delta \U^{k+1,n+1}=-\R(\U^\nk)$ with $\J$ in \eqref{eq.jacobian} \\
	Minimize $\|\R(\U^{k+1,n+1})\|$, where $\U^{k+1,n+1} = \lambda\Delta \U^\nk + \U^\nk$, with respect to $\lambda$ \\
	Set $\U^{k+1,n+1}=\lambda\Delta \U^\nk + \U^\nk$ \\
	\eIf{\rm Continuation }{
		Set $\varepsilon^k = \tilde{\varepsilon}\frac{\|\R(\U^{k+1,n+1})\|}{\|\R(\U^{0,n+1})\|}$ \\
	}{
		Set $\varepsilon^k = \varepsilon$
	}
	Set $\sigma^k = 10^2\, \varepsilon^k $ \\
	Update $k= k+1$
}
\end{algorithm}

\section{Numerical experiments}\label{sec.experiments}
In this section, we perform several numerical experiments to assess the numerical scheme introduced in the previous sections. First, we perform a convergence analysis to assess its implementation. Then, we use a steady benchmark test to analyze the effectiveness of the regularization parameters. We also analyze their effectiveness in the case of a transient problem. Finally, we solve a slightly more challenging steady benchmark test.

In all experiments below we assume that the ideal gas state equation applies, and we use an adiabatic index of $\gamma=1.4$. From previous experience \cite{badia_differentiable_2017,badia_monotonicity-preserving_2017,Bonilla2018a,Bonilla2019a}, the effects of parameters $\sigma$ and $\varepsilon$ to the nonlinear convergence and numerical error are analogous. Hence, we consider $\varepsilon = 10^{-2} \sigma$. In addition, for all the tests below, the density is discontinuous at all shocks. Therefore, we use $C=\{1\}$ in \eqref{eq.system-detector}, i.e. the shock detector is based on the density behavior.

\subsection{Convergence test}
We use two different problems to assess the convergence rate of the scheme. One has a smooth solution, whereas in the other there is a shock. The smooth problem is simply the translation of a sinusoidal perturbation in the density, with constant pressure and velocity. In particular, the solution for $r=\sqrt{(0.5+t-x)^2 + (0.5-y)^2}<0.5$ is
\begin{equation}
\contunk=\left[\begin{array}{c}
\rho \\
v_1 \\ v_2 \\ p
\end{array}\right]
= \left[\begin{array}{c}
1+0.9999 \cos (2\pi r) \\
1 \\ 0 \\ 1
\end{array}\right],
\end{equation}
and $\contunk=[0.0001, 1, 0, 1]^t$ otherwise.

The non-smooth problem is the well known compression corner test \cite{AndersonJr.1990,Kuzmin2012}, also known as oblique shock test \cite{shakib_new_1991,tezduyar_stabilization_2006}. This benchmark consists in a supersonic flow impinging to a wall at an angle. We use a $[0,1]^2$ domain with a $M=2$ flow at $10^{\circ}$ with respect to the wall. This leads to two flow regions separated by an oblique shock at 29.3$^{\circ}$, see the scheme in Fig.\ \ref{fig.corner}. 
\begin{figure}[h]
	\centering
	\includegraphics[width=0.4\textwidth]{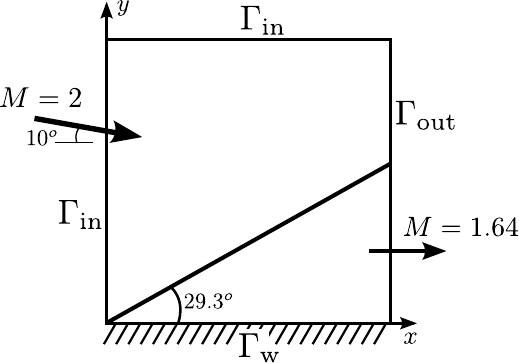}
	\caption{Compression corner scheme.}
	\label{fig.corner}
\end{figure}

For both tests, we compare the convergence rates for the differentiable and the non-differentiable schemes. $q$ is set to 10 and the regularization parameters are $\gamma=10^{-10}$, $\varepsilon=10^{-4}$, and $\sigma=10^{-2}$ in the differentiable version. The convergence criterion for both tests is $\frac{\|\Delta \bunk^{k+1}\|}{\|\bunk^k\|}<10^{-6}$. The scheme is able to converge in less than 10 iterations for the smooth problem, regardless of the setting or the mesh used. However, for the compression corner some tests did not converge. In this case, the iteration limit is set to 150. Nevertheless, $\frac{\|\Delta \bunk^{k+1}\|}{\|\bunk^k\|}\ll \|\rho-\rho_h\|_{L^1(\Omega)}$ is always checked and satisfied for all tests.

In Fig.\ \ref{fig.convergence-rates}, the $\lone$ error is depicted for different mesh sizes, and in Tab.\ \ref{tab.convergence-rates} we collect the measured convergence rates. It can be observed that for a smooth problem both settings recover second order convergence, whereas for non-smooth problems the expected first order convergence rates are obtained. For this particular choice of regularization parameters, we observe that the errors are slightly higher. However, the convergence rates are not affected by the regularization described in Sect.\ \ref{sec.differentiability}.

\begin{figure}[h]
\begin{subfigure}[b]{0.48\textwidth}
	\includegraphics[width=\textwidth]{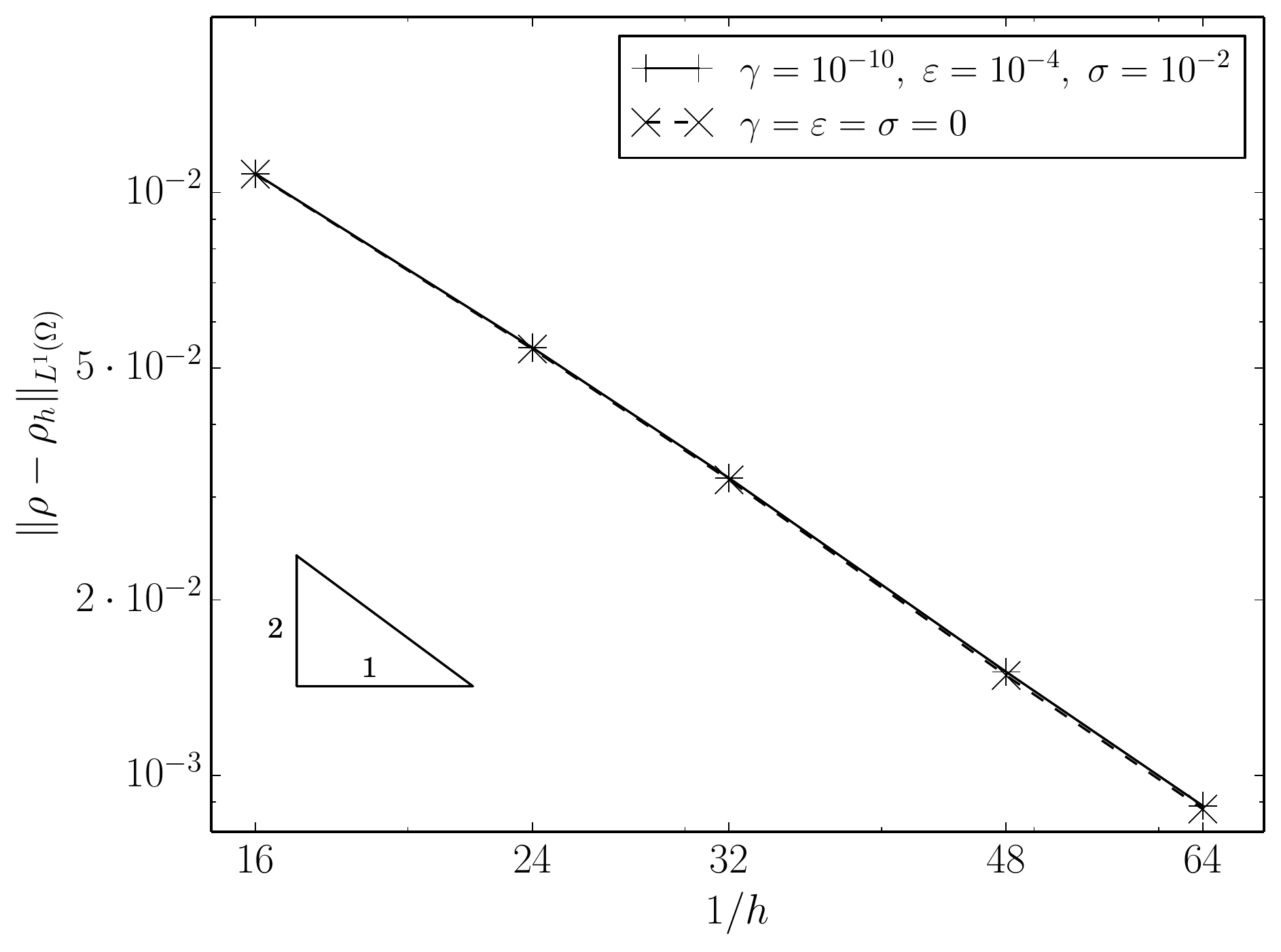}
	\caption{Sinusoidal translation test.}
\end{subfigure}
\begin{subfigure}[b]{0.48\textwidth}
	\includegraphics[width=\textwidth]{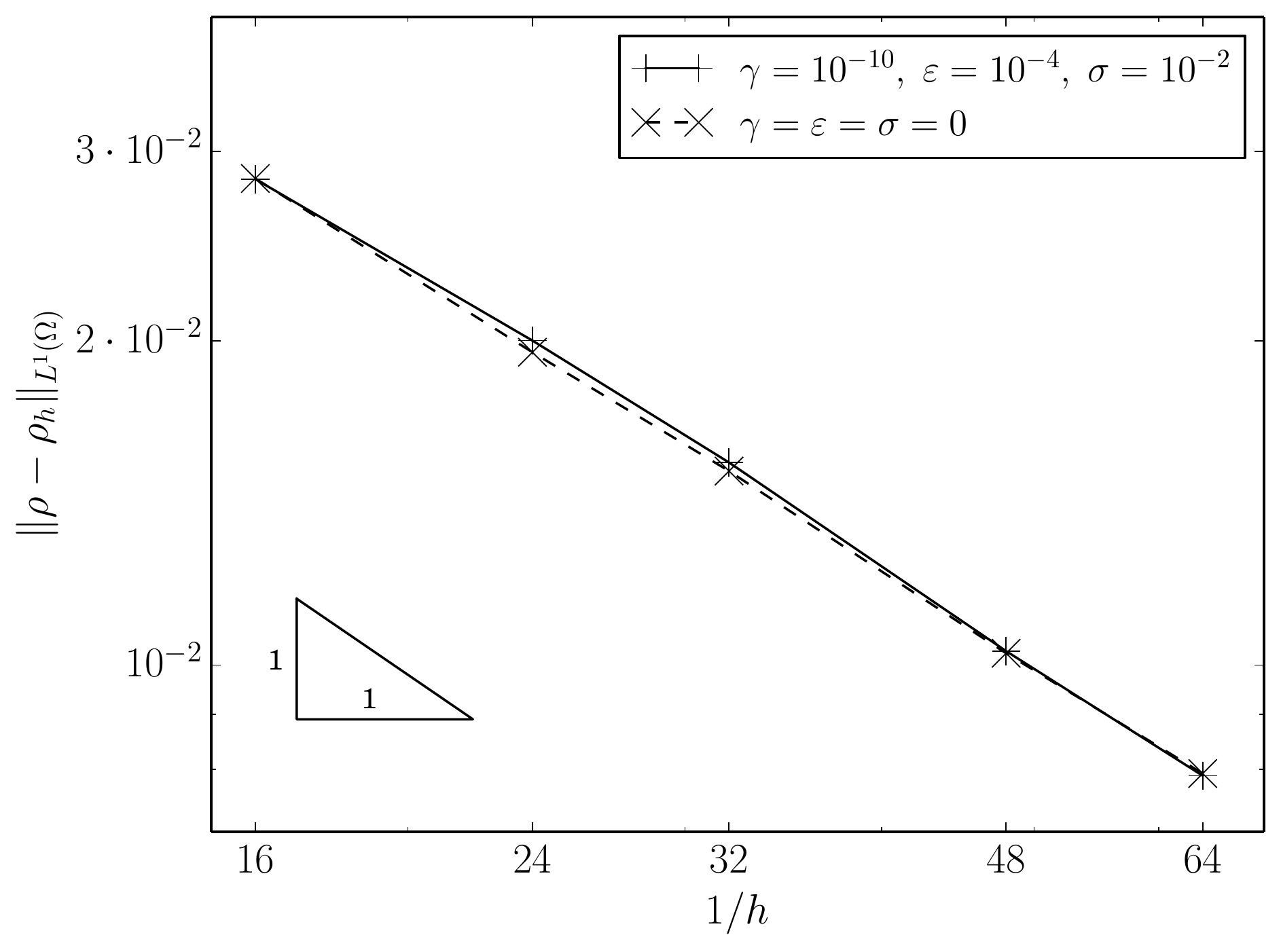}
	\caption{Compression corner test.}
\end{subfigure}
\caption{Density convergence for successive mesh refinements.}
\label{fig.convergence-rates}
\end{figure}

\begin{table}[h]
	\centering
	\caption{Experimental convergence rates for both problems.}
	\label{tab.convergence-rates}
	\begin{tabular}{lc} \hline
		\hspace{14mm} Test & $L_1$ error  \\ \hline
		Sinusoidal translation (differentiable) & 1.8099  \\
		Sinusoidal translation (non-differentiable) & 1.8190  \\
		Compression corner (differentiable) & 0.9278  \\
		Compression corner (non-differentiable) & 0.9207 \\ \hline
	\end{tabular}
\end{table}

\subsection{Reflected Shock}
In this test, we compare the nonlinear convergence behavior of the method for different regularization parameters. This benchmark consists in two flow streams colliding at different angles. The domain has dimensions $[0.0,1.0]\times[0.0,4.1]$ and a solid wall at its lower boundary. This configuration leads to a steady shock separating both flow regimes, which in turn, is reflected at the wall producing a third different flow state behind it. A sketch of this benchmark test is given in Fig.\ \ref{fig.reflected-scheme}. The flow states at each region have been collected in Tab.\ \ref{tab.reflected-shock}.

\begin{figure}[h]
	\centering
	\includegraphics[width=0.7\textwidth]{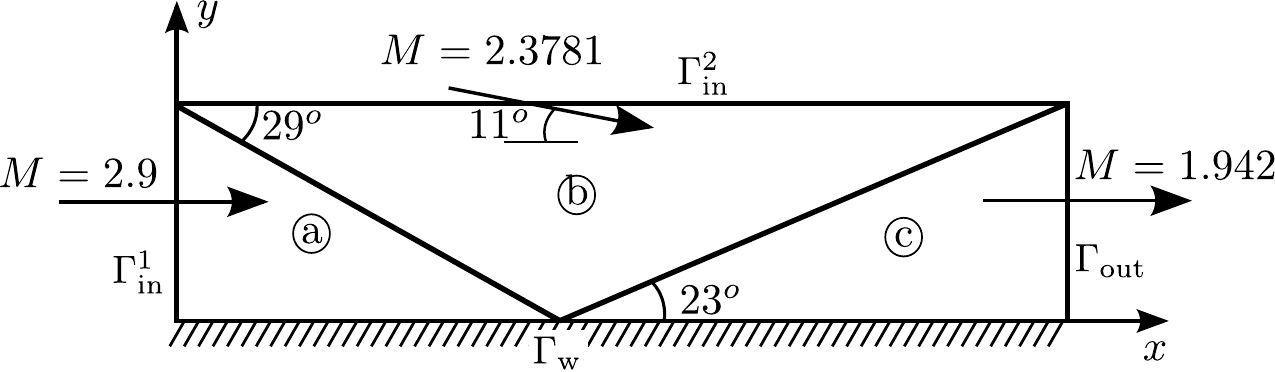}
	\caption{Reflected shock scheme.}
	\label{fig.reflected-scheme}
\end{figure}

\begin{table}[h]
\centering
\caption{Reflected shock solution values at every region.}
\label{tab.reflected-shock}
\begin{tabular}{cccc} \hline
    Region & Density [Kg\,m$^{-3}$] & Velocity [m\,s$^{-1}$] & Total energy [J] \\ \hline
	\textcircled{a} & 1.0 & (2.9, 0.0) & 5.99075 \\
	\textcircled{b} & 1.7 & (2.62, -0.506)& 5.8046 \\
	\textcircled{c} & 2.687 & (2.401, 0.0) & 5.6122 \\ \hline
\end{tabular}
\end{table}
We use a 60$\times$20 structured $\mathcal{Q}_1$ mesh. The problem is solved directly to steady state using the hybrid method and the continuation scheme described in Sect.\ \ref{sec.solver}. The tolerance used for switching from Picard to Newton linearization is $10^{-2}$. We compare the convergence behavior for $q=\{1,2,5,10\}$. For the differentiable stabilization we use the following values for $\tilde{\varepsilon}=\{10^{-4},10^{-2},1\}$. We consider $\varepsilon^k = \sigma^k\, 10^{-2}$. The value of $\gamma$ is $10^{-10}$. 

In Figs.\ \ref{fig.reflected-convergence-q1}-\ref{fig.reflected-convergence-q10} for every nonlinear iteration we depict (from left to right) the relative residual, the relative Galerkin residual, and the relative solution variation between iterations. The Galerkin residual is 
simply the residual in \eqref{eq.residual} minus the stabilization terms, i.e.,
\begin{equation}
\R^* (\bunk^k) \doteq \K(\bunk^k)\U^k - \G.
\end{equation}
We depict this value relative to the Galerkin residual of the non-differentiable scheme. This value gives a sense of how close is the computed approximation to the solution of the original problem. However, since it omits the stabilization terms present in the system solved, it will stagnate at some point.

In general, we can observe in Figs.\ \ref{fig.reflected-convergence-q1}-\ref{fig.reflected-convergence-q10} that as $q$ is increased nonlinear convergence rates are reduced. For instance, one can observe that the higher is $q$ the more iterations the scheme needs to reach $\|\R(\bunk^{k+1})\|/\|\R(\bunk^{0})\|<10^{-3}$. Unfortunately, using a low value of $q$ might also make the scheme to stagnate before reaching convergence. This is observed in Figs.\ \ref{fig.reflected-convergence-q1} and \ref{fig.reflected-convergence-q2}. 
In addition, we observe a $15\%$ to $35\%$ reduction in the number of iterations when the differentiable scheme is used. Another interesting observation is about the behavior of the Galerkin residual. This residual is not expected to converge since we are not solving the original problem but the stabilized one. However, it provides an indication of how close to the original solution is the one obtained by the proposed scheme. During the first iterations, the differentiable scheme is able to provide solutions closer to the solution of the original problem. This implies that, up to some extent, the differentiable scheme is able to provide more accurate solutions from the beginning of the iterative process. It is also interesting to observe the improvement in the residual convergence once the complete Jacobian is used, i.e. after $\|\R(\bunk^{k+1})\|/\|\R(\bunk^{0})\|<10^{-2}$. This is specially evident in Figs.\ \ref{fig.reflected-convergence-q5}-\ref{fig.reflected-convergence-q10}.

\begin{figure}[h]
	\centering
	\includegraphics[width=\textwidth]{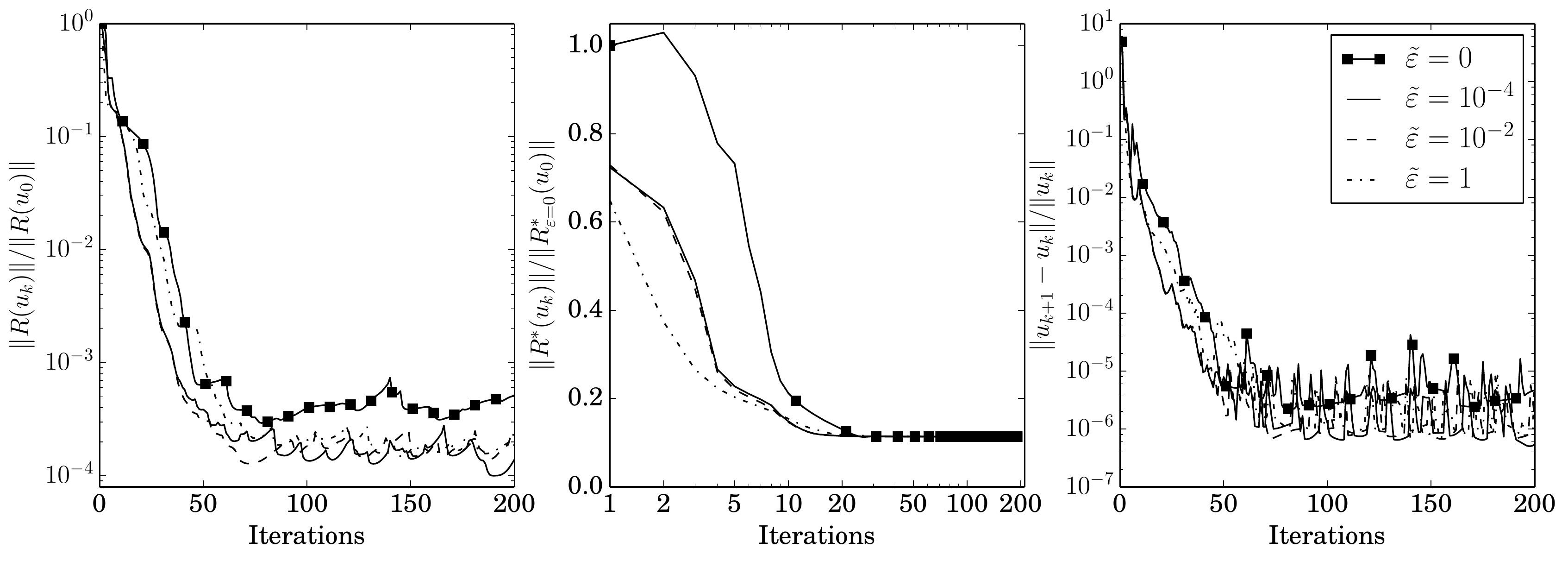}
	\caption{Reflected shock convergence history for $q=1$.}
	\label{fig.reflected-convergence-q1}
\end{figure}
\begin{figure}[h]
	\centering
	\includegraphics[width=\textwidth]{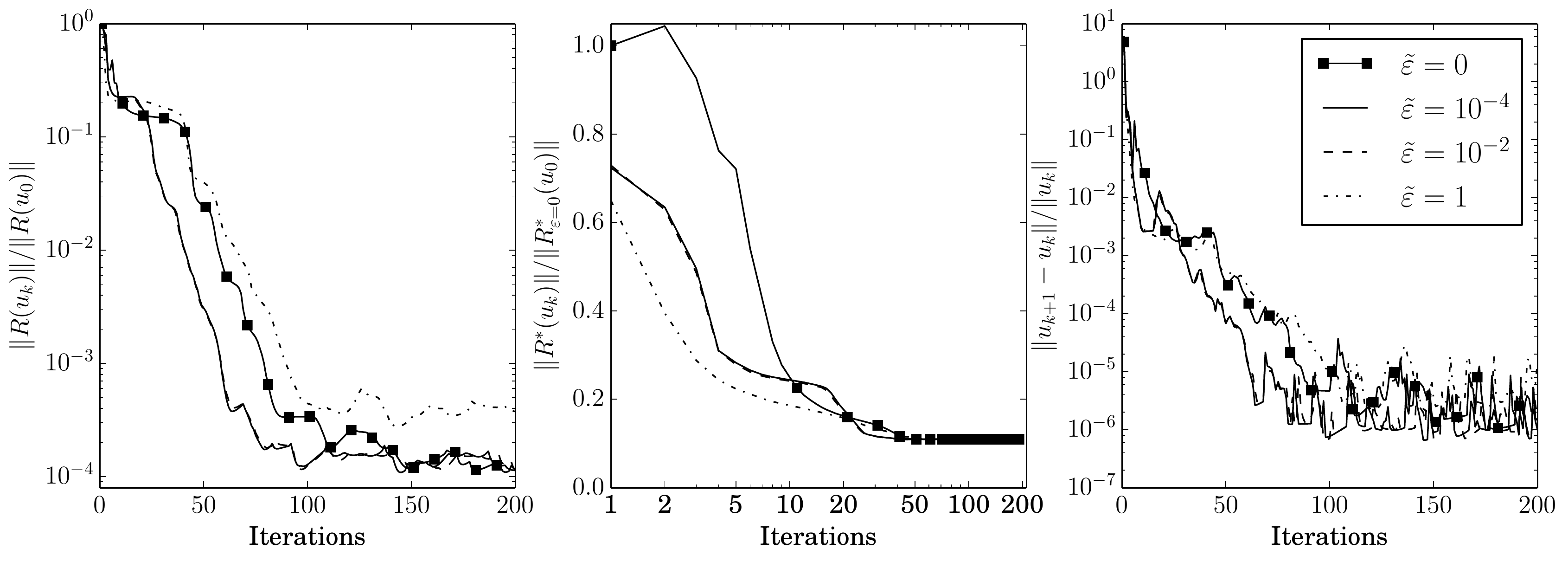}
	\caption{Reflected shock convergence history for $q=2$.}
	\label{fig.reflected-convergence-q2}
\end{figure}
\begin{figure}[h]
	\centering
	\includegraphics[width=\textwidth]{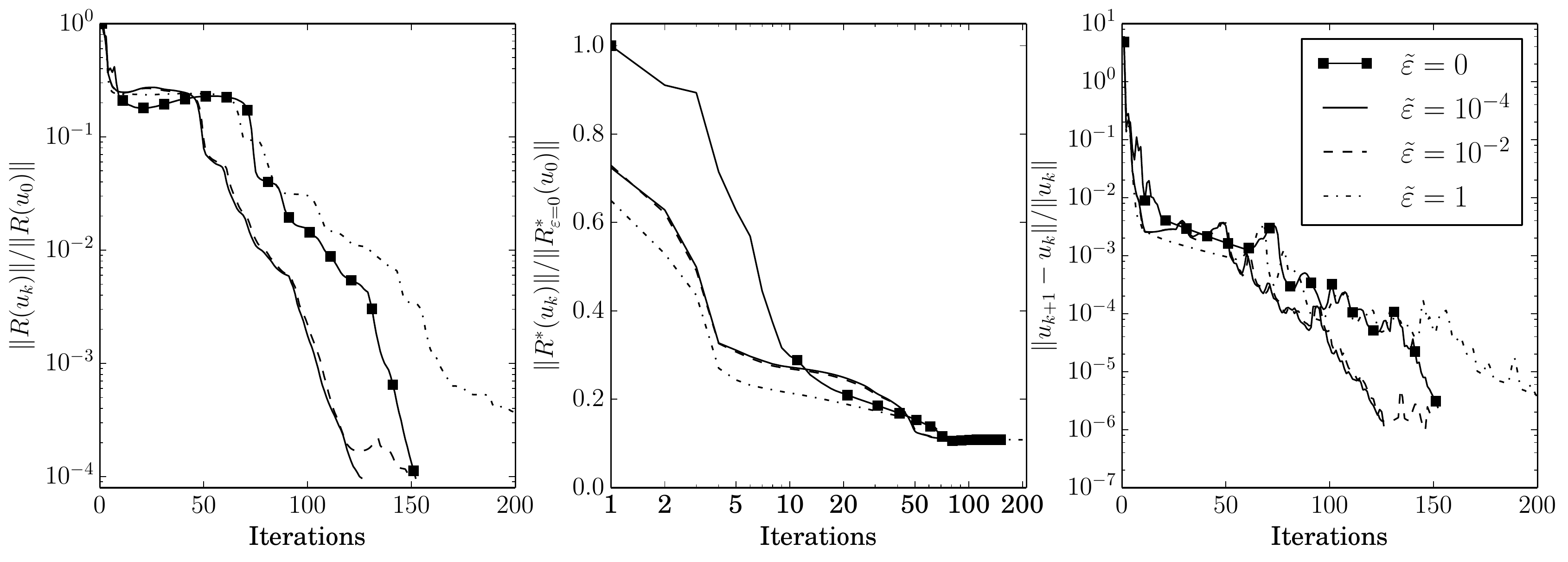}
	\caption{Reflected shock convergence history for $q=5$.}
	\label{fig.reflected-convergence-q5}
\end{figure}
\begin{figure}[h]
	\centering
	\includegraphics[width=\textwidth]{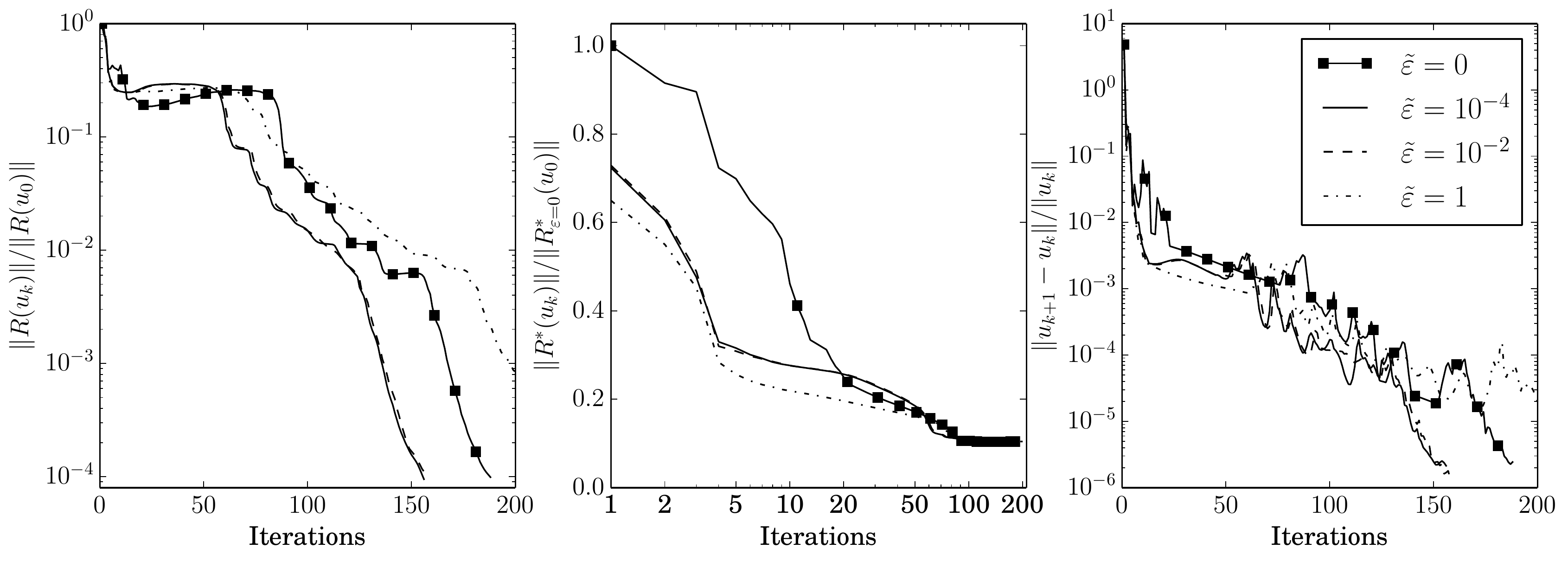}
	\caption{Reflected shock convergence history for $q=10$.}
	\label{fig.reflected-convergence-q10}
\end{figure}

\subsection{Sod's Shock Tube} 
In this section, we evaluate the effect of the differentiability in the case of a transient problem. To this end, we solve the well known Sod's shock tube test. This is a one dimensional problem that assesses the evolution of a fluid initially at rest with a discontinuity in density and pressure. The discontinuity is initially placed at $x=0.5$. Even though it is a 1D test, we consider a narrow 2D strip of dimensions $[0,1]\times[0,0.01]$ and we let the problem evolve until $t=0.2$.  We use a $\mathcal{Q}_1$ \ac{fe} mesh of size $\Delta x=0.01$ and a time step length of $\Delta t=0.001$. Initial conditions at the left of the discontinuity are $\boldsymbol{u}_0 = (1, 0, 0, 2.5)$ and at the right $\boldsymbol{u}_0 = (0.125, 0, 0, 0.25)$. See the initial condition depicted in Fig.\ \ref{sfig.sod-shoc-ic}.

In this case, the hybrid nonlinear solver described in Sect.\ \ref{sec.solver} is used directly without the continuation scheme. The tolerance used for switching from the Picard to Newton linearization is $5\cdot 10^{-3}$. We set the nonlinear convergence criteria in terms of the relative residual, namely 
$
\frac{\|\R(\bunk^\nk)\|}{\|\R(\bunk^{0,n+1})\|} < 10^{-6}
$.
We use $\gamma=10^{-10}$, $\varepsilon=\{10^{-2},10^{-3},10^{-4},10^{-5}\}$, and $\varepsilon=\sigma\,10^{-2}$ for the differentiable stabilization. We also use different values of $q$ for this comparison, namely $q=\{1, 2, 4, 6, 8, 10, 12\}$. 

\begin{figure}[h]
	\centering
	\begin{subfigure}[b]{0.45\textwidth}
		\includegraphics[width=\textwidth]{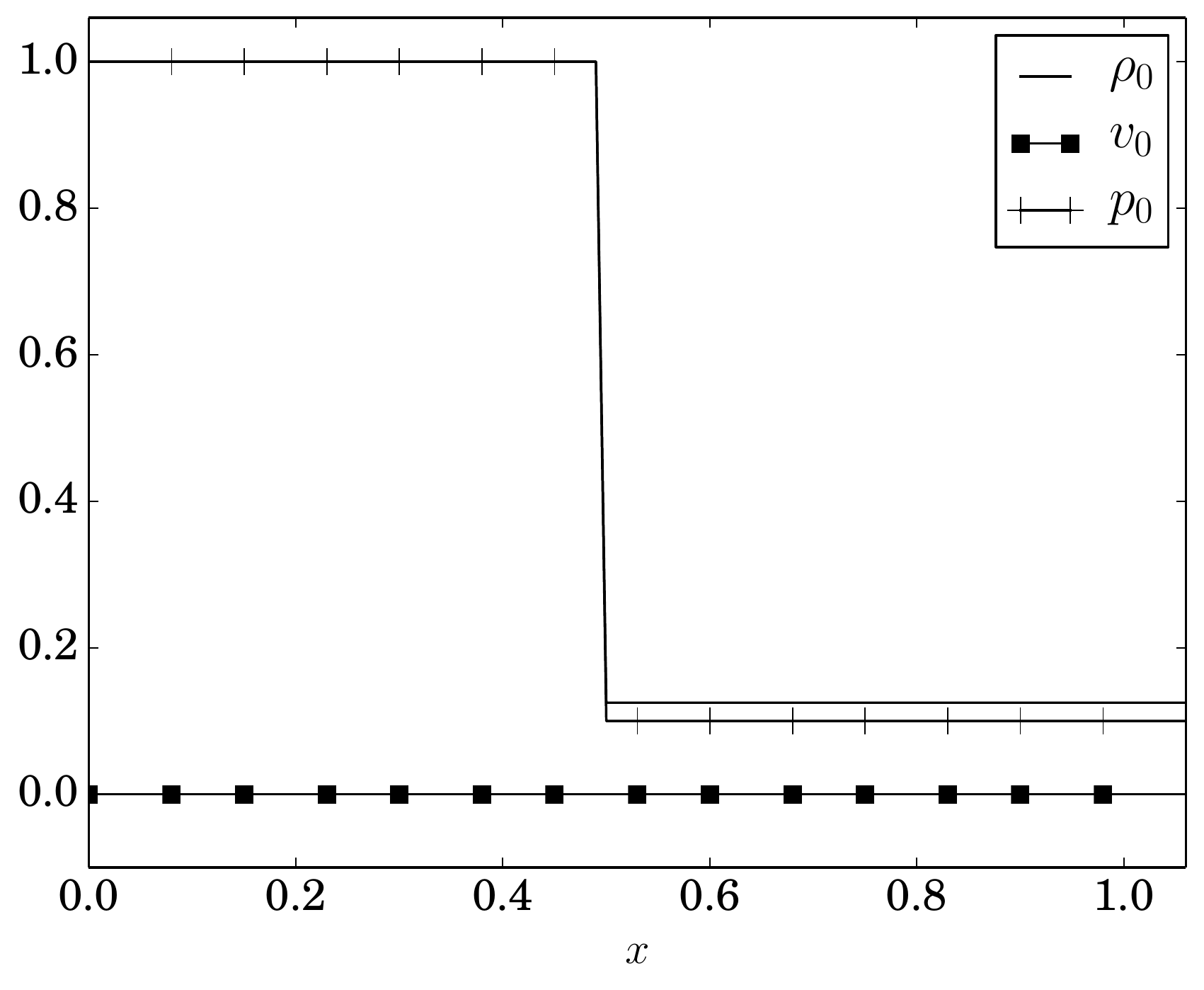}
		\caption{Initial condition.}
		\label{sfig.sod-shoc-ic}
	\end{subfigure}
	\begin{subfigure}[b]{0.45\textwidth}
		\includegraphics[width=\textwidth]{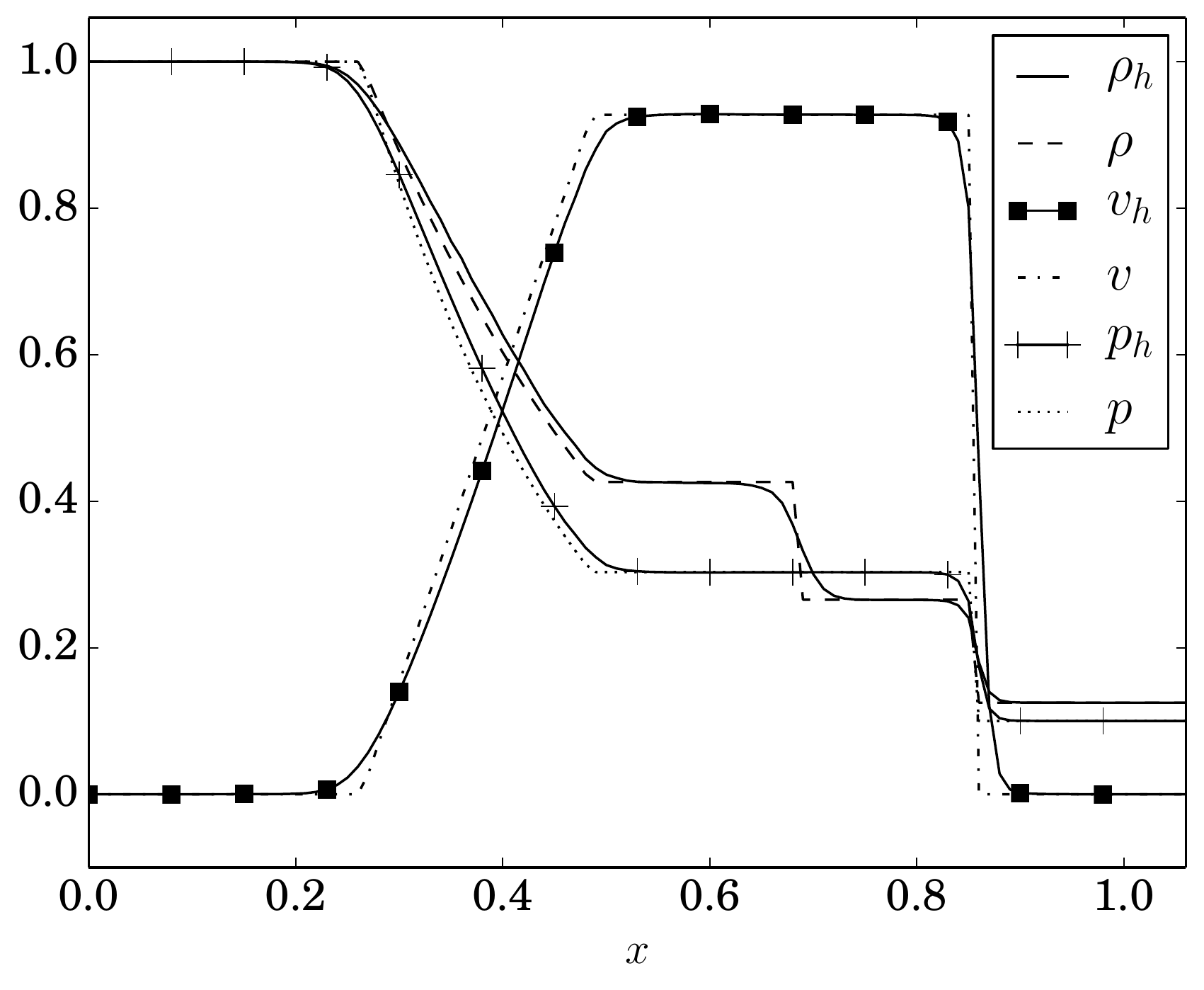}
		\caption{Solution a $t=0.2$.}
		\label{sfig.sod-shoc-q10}
	\end{subfigure}
	\caption{Sod shock initial condition and solution for the differentiable scheme using parameters $q=10$, $\sigma=10^{-3}$, $\varepsilon=10^{-5}$, and $\gamma=10^{-10}$.}
	\label{fig.sod-shock_solution}
\end{figure}
Fig.\ \ref{sfig.sod-shoc-q10} shows a comparison at $t=0.2$ of the exact solution from ExactPack \cite{Singleton2017} against the obtained solution for $q=10$, $\sigma=10^{-3}$, $\varepsilon=10^{-5}$, and $\gamma=10^{-10}$. In this case, we observe a good agreement of the obtained solution despite the rather coarse mesh being used.

In Fig.\ \ref{fig.sod-shock}, for different regularization values, we depict the total number of nonlinear iterations required to reach $t=0.2$, and the density error $\lone$ norm, as a function of the value of $q$. For each chart, we compare the results for the differentiable and non-differentiable stabilization. Analyzing these figures, several general observations can be made. One recovers the behavior of the non-differentiable scheme as the parameters used in the differentiable scheme become smaller (see \ref{sfig.sod-shoc-4}). Using large values for the regularization parameters improves the computational cost required at the expense of higher numerical errors (see \ref{sfig.sod-shoc-1}). It can also be seen that for transient problems the benefits of differentiability are not as evident as for problems solved directly to steady state. Notice that the differentiable scheme always require less iterations to converge. However the difference is smaller than for steady state problems as one would expect since the time advancement provides a decrease in the effective regularization by the existence of the time derivative and the evolution of the transient problem at each time step.

Another interesting observation can be made when moderate values for the parameters are used. Namely, the differentiable scheme is able to yield results with a similar accuracy while requiring a lower computational cost. For example, let us focus on Fig.\ \ref{sfig.sod-shoc-2} and compare the performance of the differentiable scheme for $q=2$ and the non-differentiable with $q=1$. The first observation is that both settings have similar accuracy. However, the differentiable scheme converges faster. If we focus on Fig.\ \ref{sfig.sod-shoc-3} similar observations can be made. For instance, compare the performance of the differentiable scheme with $q=4$ and the non-differentiable with $q=2$. The differentiable scheme is able to yield a more accurate solution in less iterations. Therefore, one can come to the conclusion that in order to achieve a given accuracy it is preferable to use the differentiable scheme with a slightly larger value of $q$ rather than the non-differentiable scheme and a lower value for $q$.
\begin{figure}[h]
	\centering
	\begin{subfigure}[b]{0.45\textwidth}
		\includegraphics[width=\textwidth]{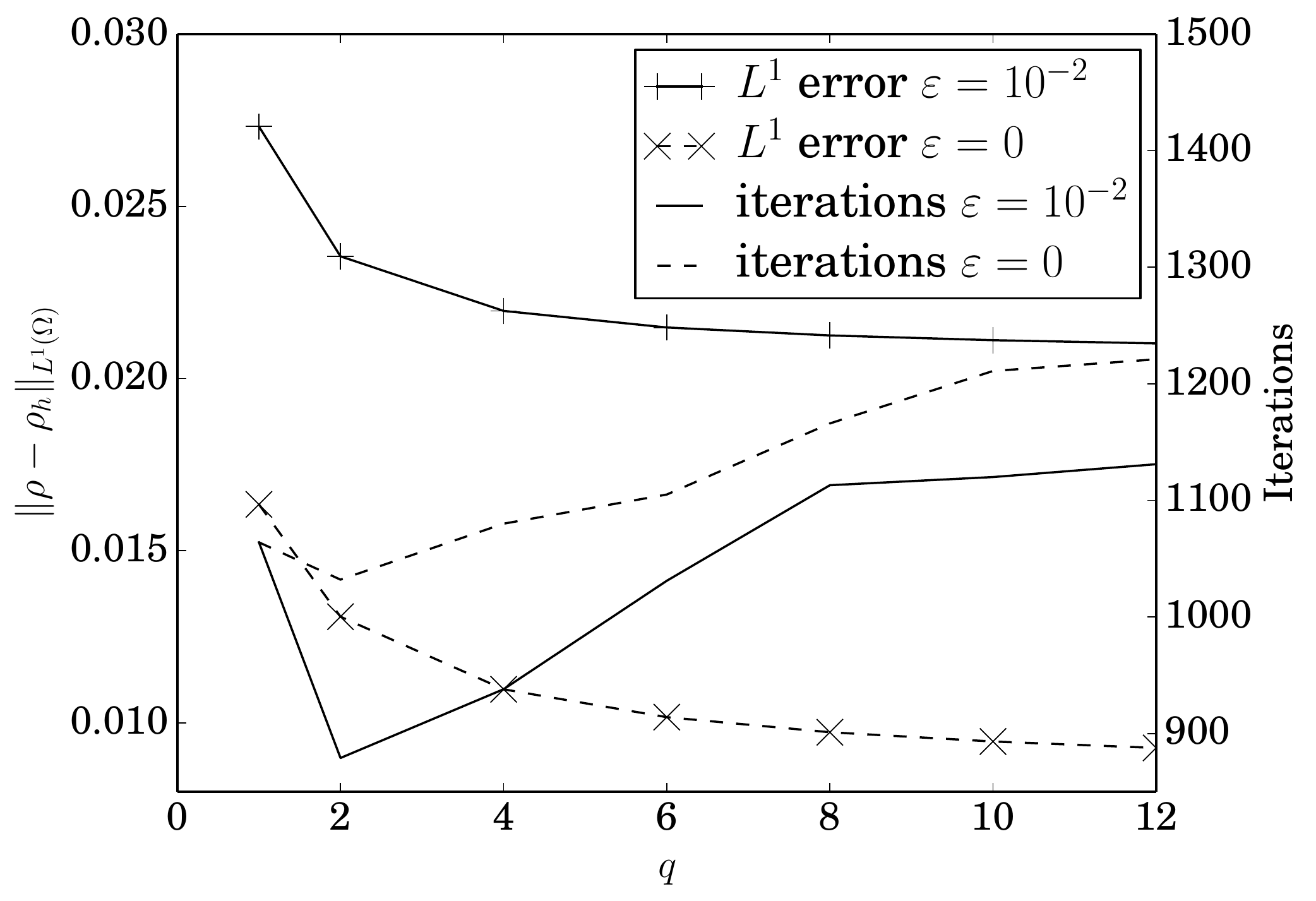}
		\caption{$\sigma=1$, $\varepsilon=10^{-2}$, $\gamma=10^{-10}$}
		\label{sfig.sod-shoc-1}
	\end{subfigure}
	\begin{subfigure}[b]{0.45\textwidth}
		\includegraphics[width=\textwidth]{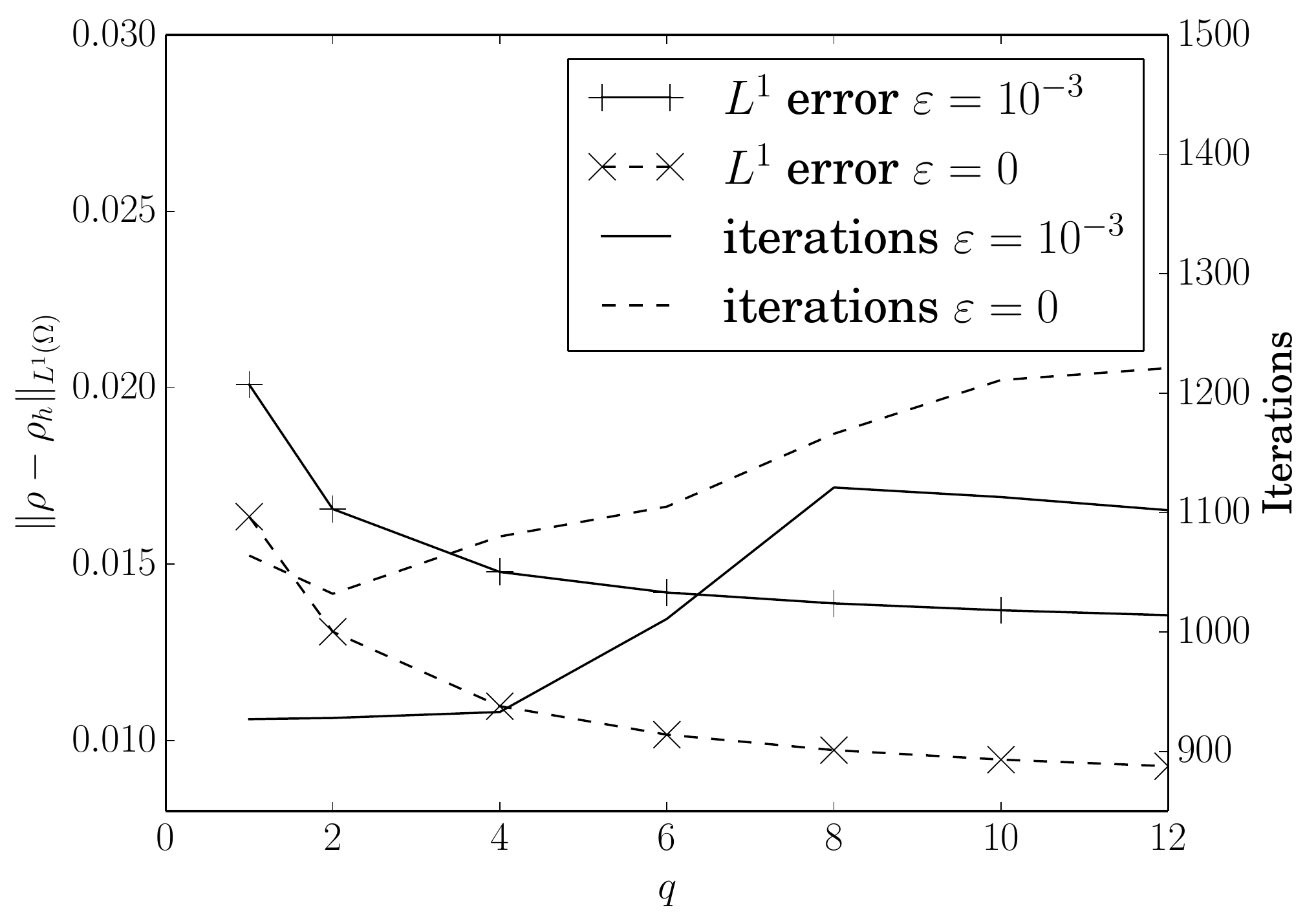}
		\caption{$\sigma=10^{-1}$, $\varepsilon=10^{-3}$, $\gamma=10^{-10}$}
		\label{sfig.sod-shoc-2}
	\end{subfigure}
	\begin{subfigure}[b]{0.45\textwidth}
		\includegraphics[width=\textwidth]{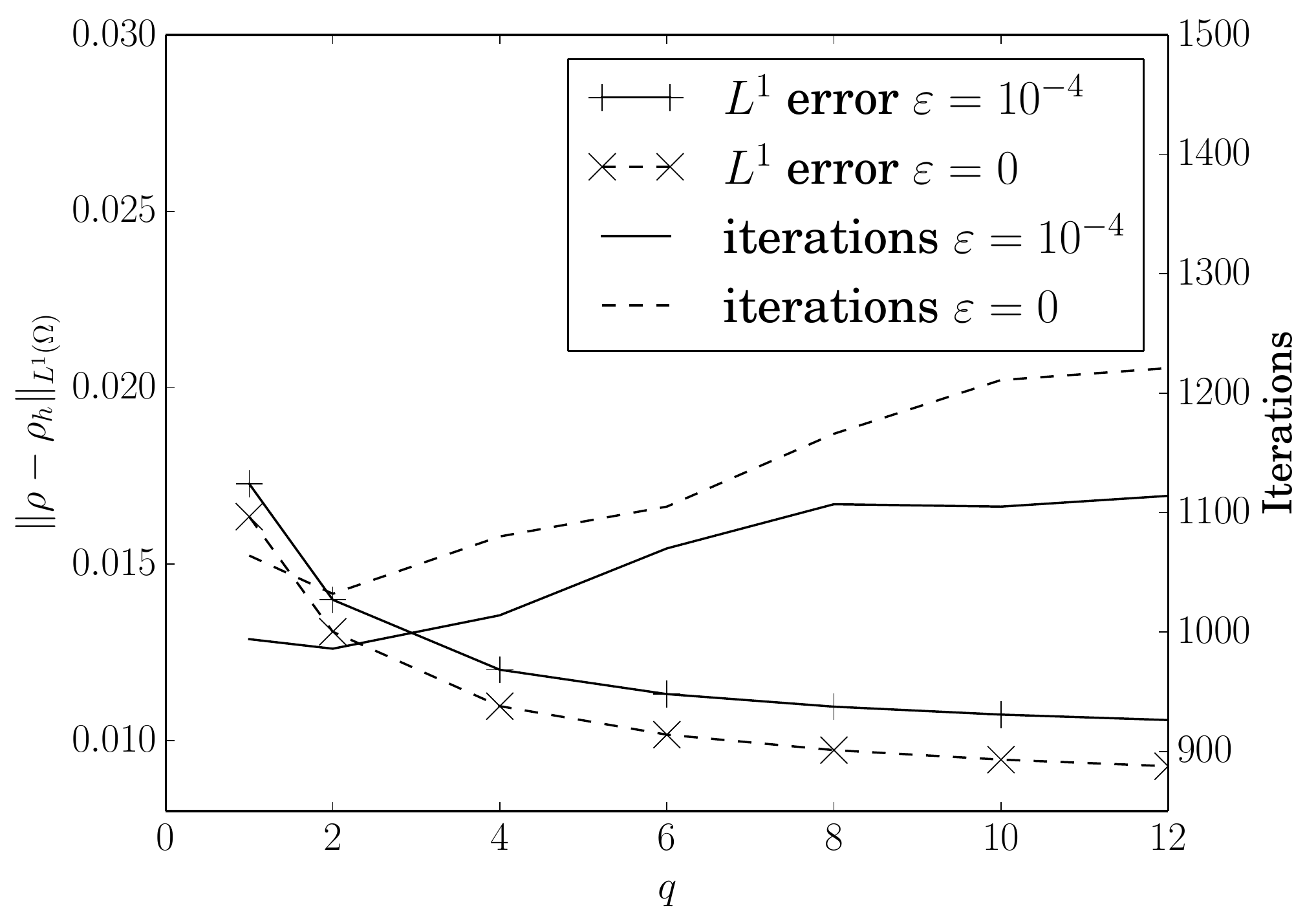}
		\caption{$\sigma=10^{-2}$, $\varepsilon=10^{-4}$, $\gamma=10^{-10}$}
		\label{sfig.sod-shoc-3}
	\end{subfigure}
	\begin{subfigure}[b]{0.45\textwidth}
		\includegraphics[width=\textwidth]{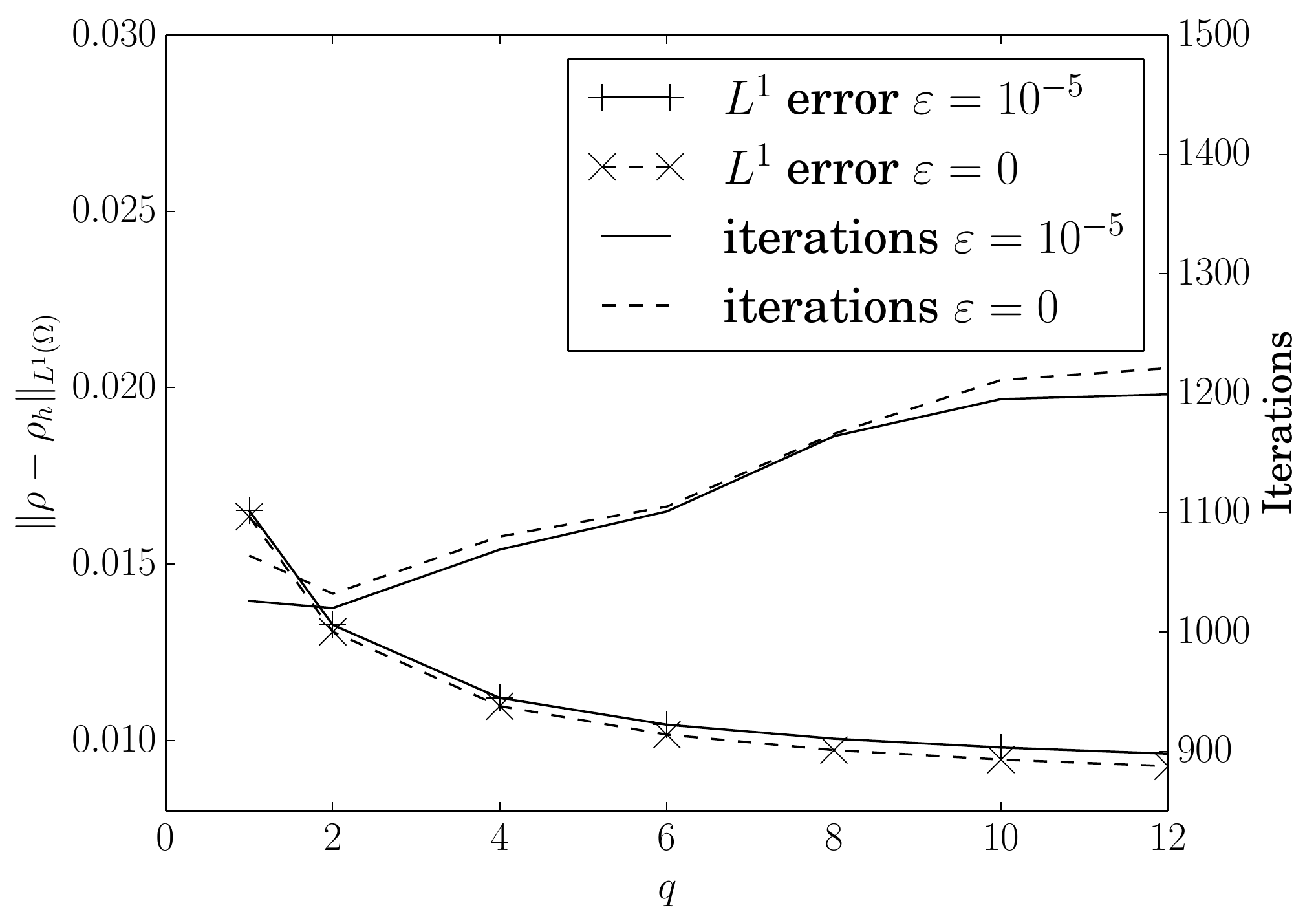}
		\caption{$\sigma=10^{-3}$, $\varepsilon=10^{-5}$, $\gamma=10^{-10}$}
		\label{sfig.sod-shoc-4}
	\end{subfigure}
	\caption{Comparison of $L^1$ error and computational cost (total number of iterations) for different regularization parameters choices at the Sod's shock test.}
	\label{fig.sod-shock}
\end{figure}

\subsection{Scramjet}
Finally, we solve a problem with a supersonic flow that develops a complex shock pattern. This test consists of a $M=3$ channel that narrows along the streamline and has two internal obstacles. In particular, Fig.\ \ref{fig.scramjet-scheme} is an illustration of the domain and Tab.\ \ref{tab.scramjet} lists the coordinates of the points defining the domain. The problem is solved directly to steady state, and two different meshes have been used. The coarsest mesh used has 18476 $\mathcal{Q}_1$ elements and the finest mesh has 63695 $\mathcal{Q}_1$ elements.

\begin{figure}[h]
	\centering
	\includegraphics[width=0.67\textwidth]{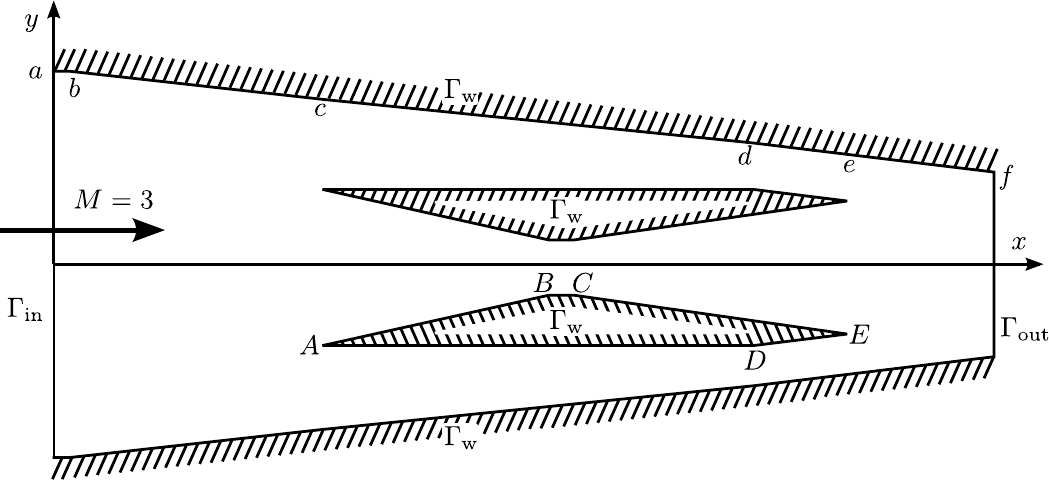}
	\caption{Scramjet test scheme.}
	\label{fig.scramjet-scheme}
\end{figure}
\begin{table}[h]
\centering
\caption{Domain coordinates for the scramjet test.}
\label{tab.scramjet}
\begin{tabular}{ccccccc} \hline
	Wall  & $a$ & $b$ & $c$ &  $d$ &   $e$ & $f$ \\ \hline
	$x_i$ & 0.0 & 0.4 & 4.9 & 12.6 & 14.25 & 16.9 \\
	$y_i$ & 3.5 & 3.5 & 2.9 & 2.12 &  1.92 &  1.7 \\  \hline   \hline 
	\multicolumn{2}{c}{Interior obstacle} & $A$ & $B$ & $C$ & $D$ & $E$ \\ \hline 
    $x_i$ & & 4.9 & 8.9 & 9.4 & 12.6 & 14.25  \\
	$y_i$ & & -1.4 & -0.5 & -0.5 & -1.4 & -1.2 \\ \hline 
\end{tabular}
\end{table}

In order to solve this problem, the hybrid nonlinear solver described in Sect.\ \ref{sec.solver} is used with the help of the continuation scheme. The tolerance for switching from the Picard to Newton linearization is set to $5\cdot 10^{-2}$. The nonlinear convergence criterion for this benchmark is 
$
\frac{\|\Delta(\bunk^{k+1})\|}{\|\bunk^k\|} < 10^{-6}
$. We also set a maximum number of iterations of 500. 
In this test, we use $q=\{2,5\}$, $\gamma=10^{-10}$, $\tilde{\varepsilon}=\{1,10^{-2},10^{-4}\}$, and $\varepsilon^k=\sigma^k\,10^{-2}$.
Even though $\sigma=10^2$ might seem a high value, we recall that it is used in the context of a continuation method. Therefore, the effective value of $\sigma^k$ is lower than 1 for the converged solution. Moreover, the actual value used in \eqref{eq:smax} is computed using the relations in \eqref{eq.scaling}.

Figs.\ \ref{fig.scramjetq5-mach}-\ref{fig.scramjetq5-rho} show, respectively, the Mach and density contours for the fine mesh, $q=5$, $\tilde{\varepsilon}=1$, $\tilde{\varepsilon}=10^2$, and $\gamma=10^{-10}$. The nonlinear convergence history for this configuration is depicted in Figs.\ \ref{sfig.scramjetq5-3}-\ref{sfig.scramjetq5-4}. The obtained values for the Mach number and the density are comparable to those in \cite{Mabuza2018,Kuzmin2012}. The shocks are well resolved. Even when using $q=2$ the shocks are properly resolved and only slightly more smeared than for $q=5$, see Fig.\ \ref{fig.scramjetq2-mach}. If instead, the coarse mesh is used (see Fig.\ \ref{fig.scramjet-coarse-q5-mach}), the solution is more dissipative. However, the scheme is able to capture most of the features present in the solution.

\begin{figure}[h]
	\centering
	\includegraphics[width=0.86\textwidth]{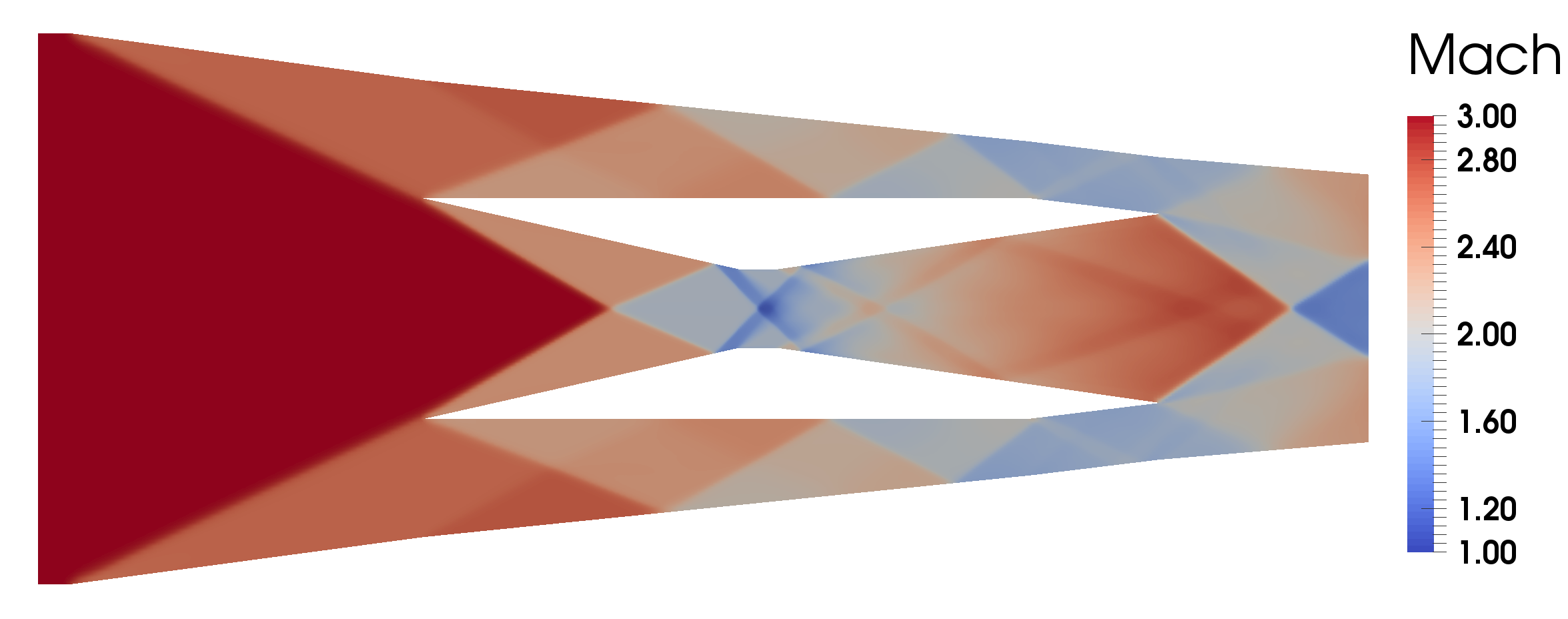}
	\caption{Scramjet Mach contours when a mesh of 63695 $\mathcal{Q}_1$ elements is used, with parameters $q=5$, $\gamma=10^{-10}$, and $\tilde{\varepsilon}=1$.}
	\label{fig.scramjetq5-mach}
\end{figure}
\begin{figure}[!h]
	\centering
	\includegraphics[width=0.86\textwidth]{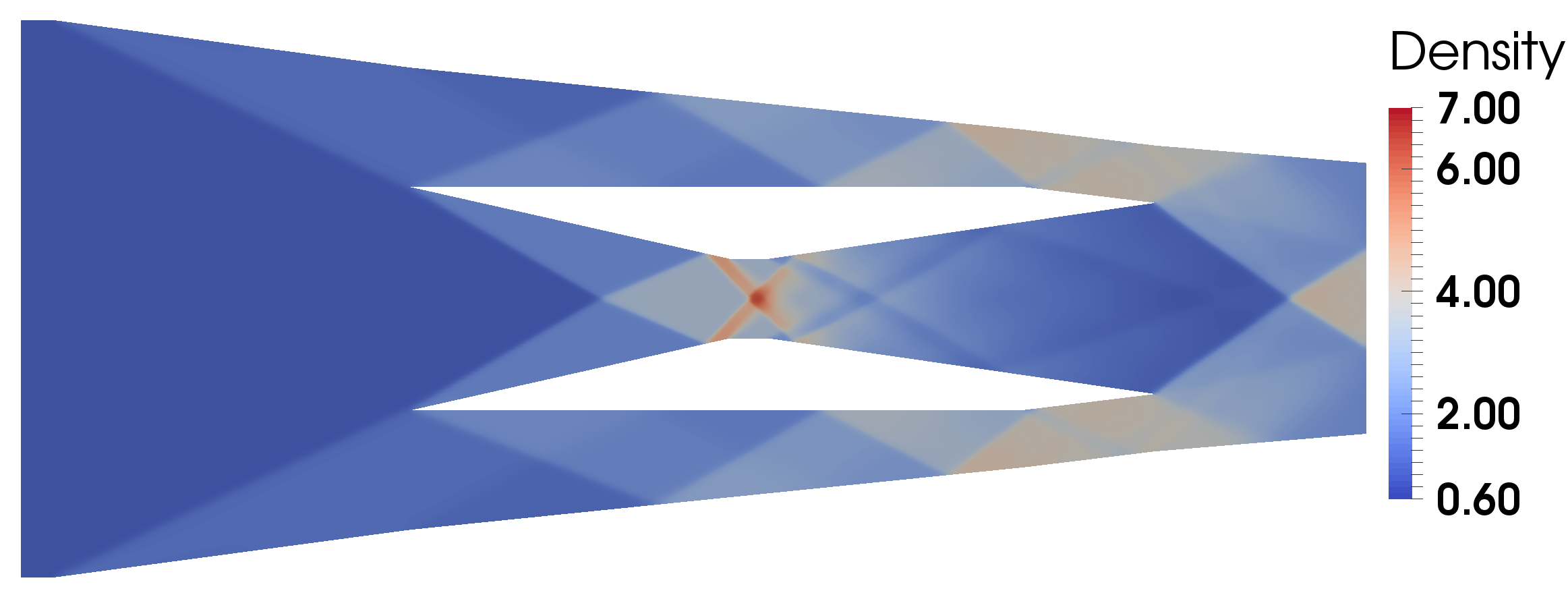}
	\caption{Scramjet density contours when a mesh of 63695 $\mathcal{Q}_1$ elements is used, with parameters $q=5$, $\gamma=10^{-10}$, and $\tilde{\varepsilon}=1$.}
	\label{fig.scramjetq5-rho}
\end{figure}
\begin{figure}[!h]
	\centering
	\includegraphics[width=0.86\textwidth]{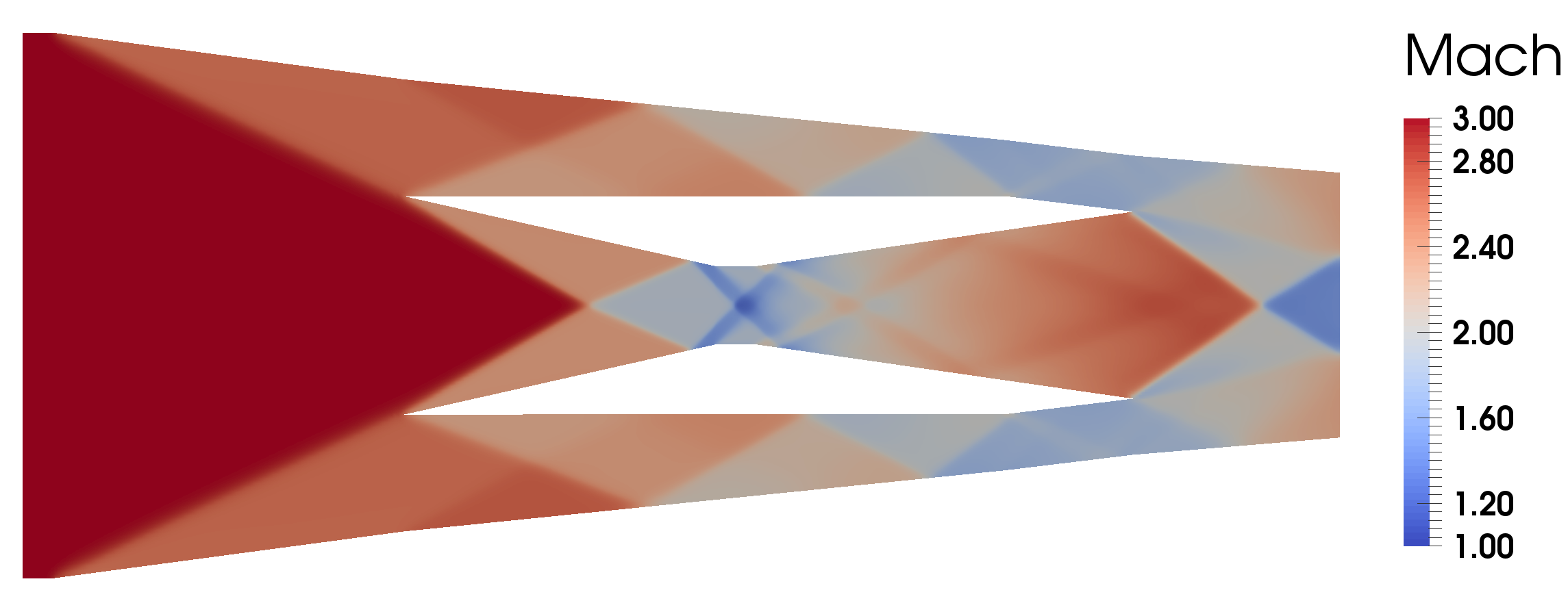}
	\caption{Scramjet Mach contours when a mesh of 63695 $\mathcal{Q}_1$ elements is used, with parameters $q=2$, $\gamma=10^{-10}$, and $\tilde{\varepsilon}=1$.}
	\label{fig.scramjetq2-mach}
\end{figure}
\begin{figure}[!h]
	\centering
	\includegraphics[width=0.86\textwidth]{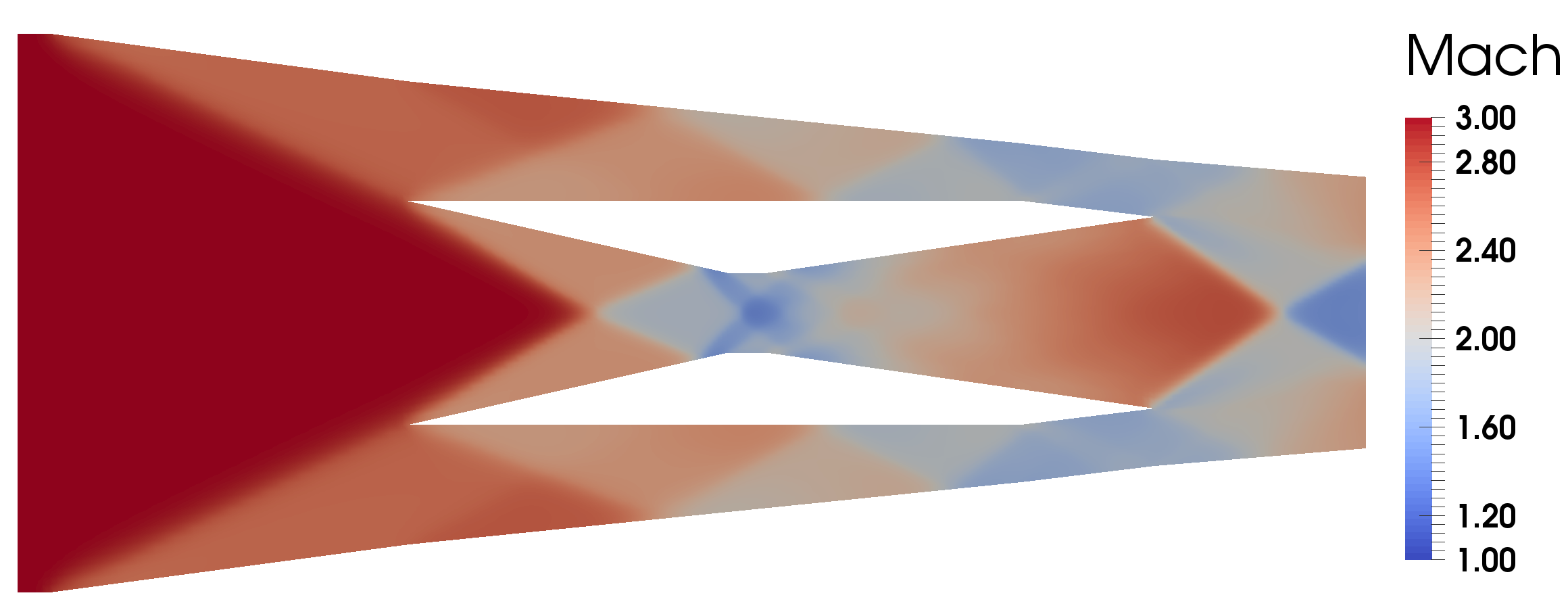}
	\caption{Scramjet Mach contours when a mesh of 18476 $\mathcal{Q}_1$ elements is used, with parameters $q=2$, $\gamma=10^{-10}$, and $\tilde{\varepsilon}=1$.}
	\label{fig.scramjet-coarse-q5-mach}
\end{figure}

Figs.\ \ref{fig.scramjet-coarse}-\ref{fig.scramjet-fine} show the nonlinear convergence history in terms of the relative residual reduction and the relative solution increment between iterations. We can observe that the convergence is not ensured for an arbitrary choice of the regularization parameters. In fact, only the tests that use $\tilde{\varepsilon}=1$ do not diverge for $q=5$ , regardless of the mesh used. Therefore, we can see that increasing the values of the regularization parameters not only improves the convergence, but also the robustness of the method. 

However, it is important to mention that even if we can improve the convergence behavior of these types of methods, this is not enough for directly solving to steady state problems with complex shock patterns. For instance, even if the solution of Fig.\ \ref{fig.scramjet-coarse-q5-mach} seems to be correct, the scheme was unable to converge to the desired tolerance (see Figs.\ \ref{sfig.scramjetq5-1} and \ref{sfig.scramjetq5-2}). However the ability to introduce differentiability into the definition of the shock detector, for robustness and increased nonlinear convergence rates, could be coupled with popular pseudo-time stepping approaches \cite{Kelley1998,Smith2004} to pursue improved methods for complex shock type systems.

\begin{figure}[h]
	\centering
	\begin{subfigure}[b]{0.45\textwidth}
		\includegraphics[width=\textwidth]{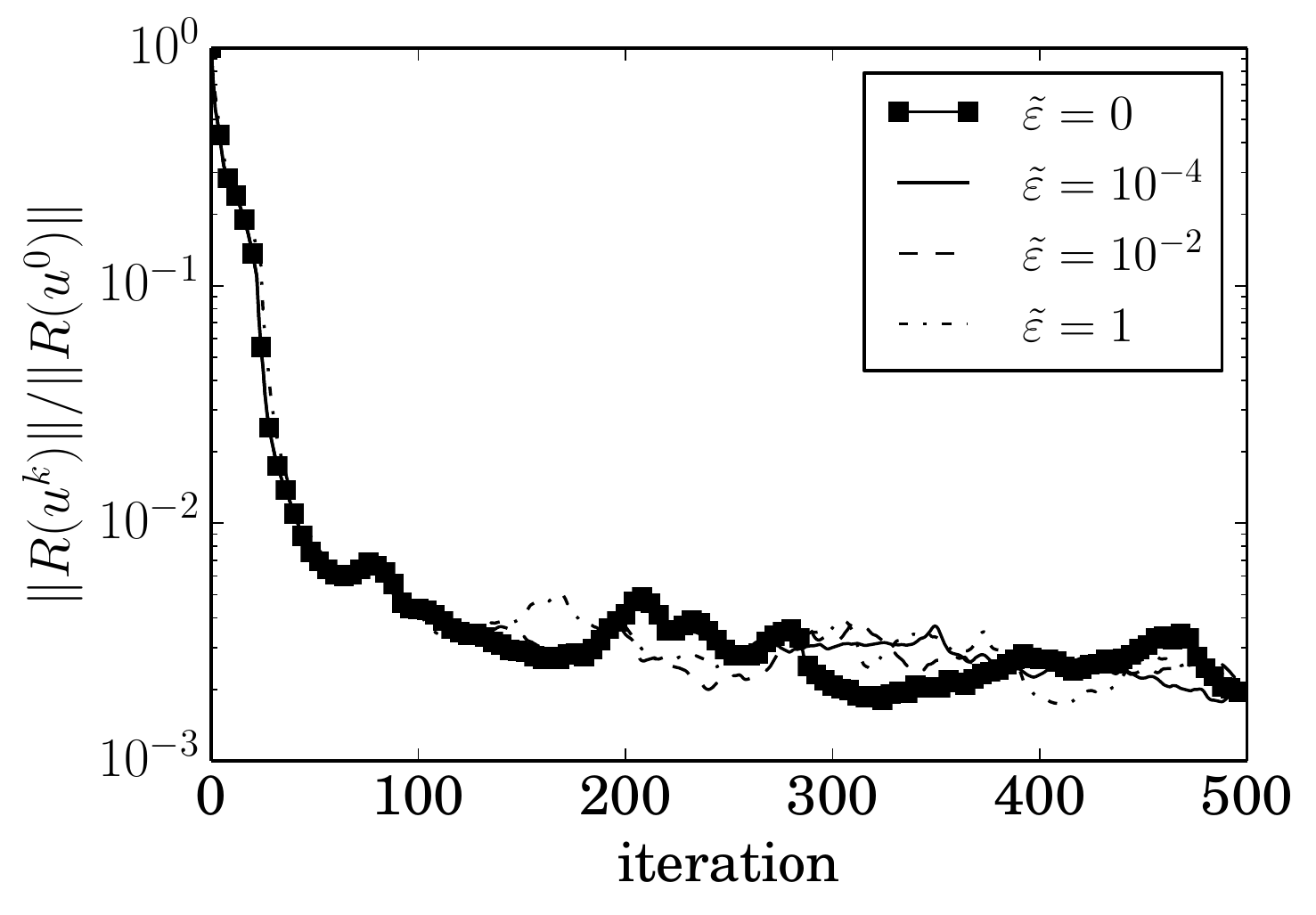}
		\caption{$q=2$.}
		\label{sfig.scramjetq2-1}
	\end{subfigure}
	\begin{subfigure}[b]{0.45\textwidth}
		\includegraphics[width=\textwidth]{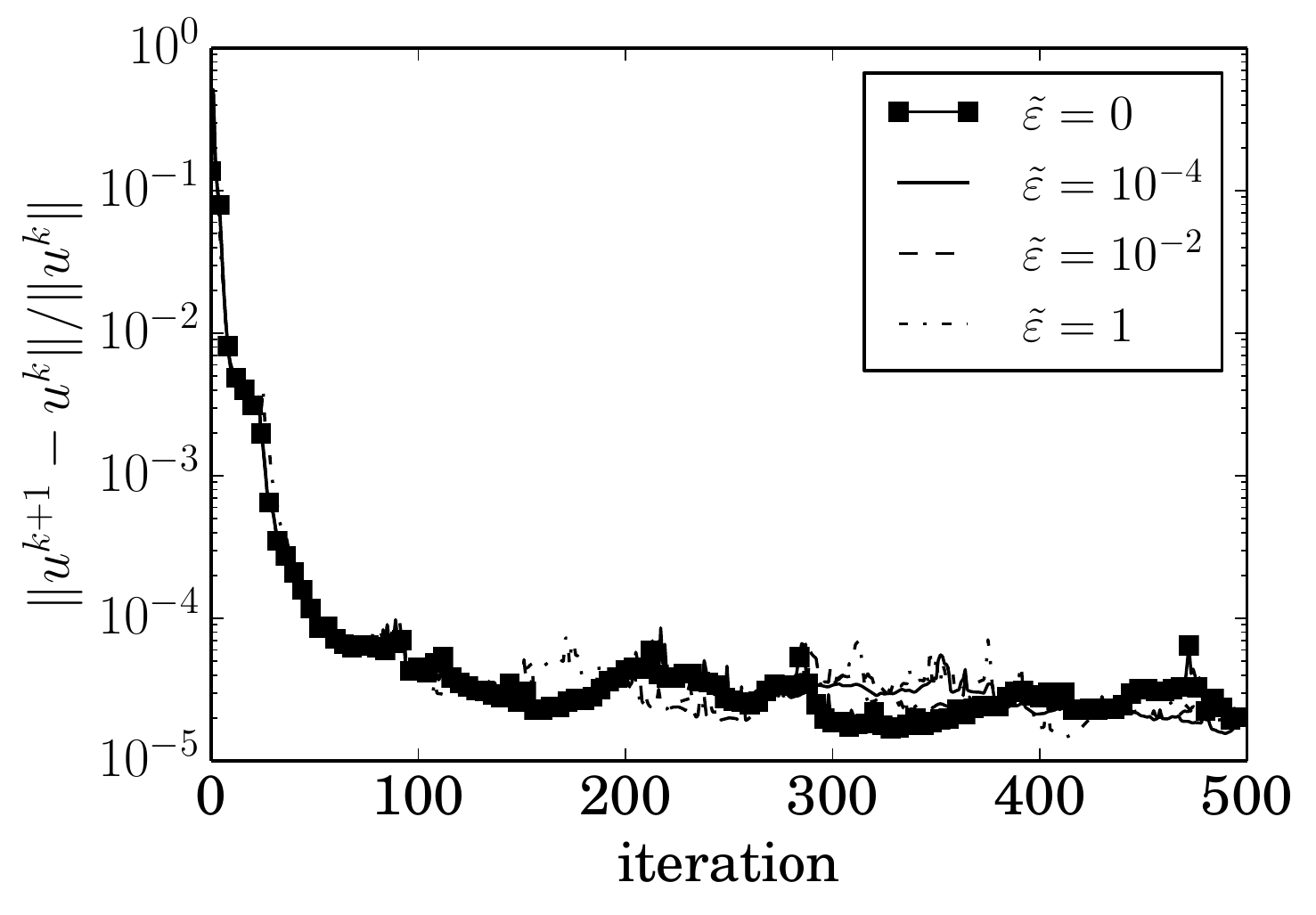}
		\caption{$q=2$.}
		\label{sfig.scramjetq2-2}
	\end{subfigure}
	\begin{subfigure}[b]{0.45\textwidth}
		\includegraphics[width=\textwidth]{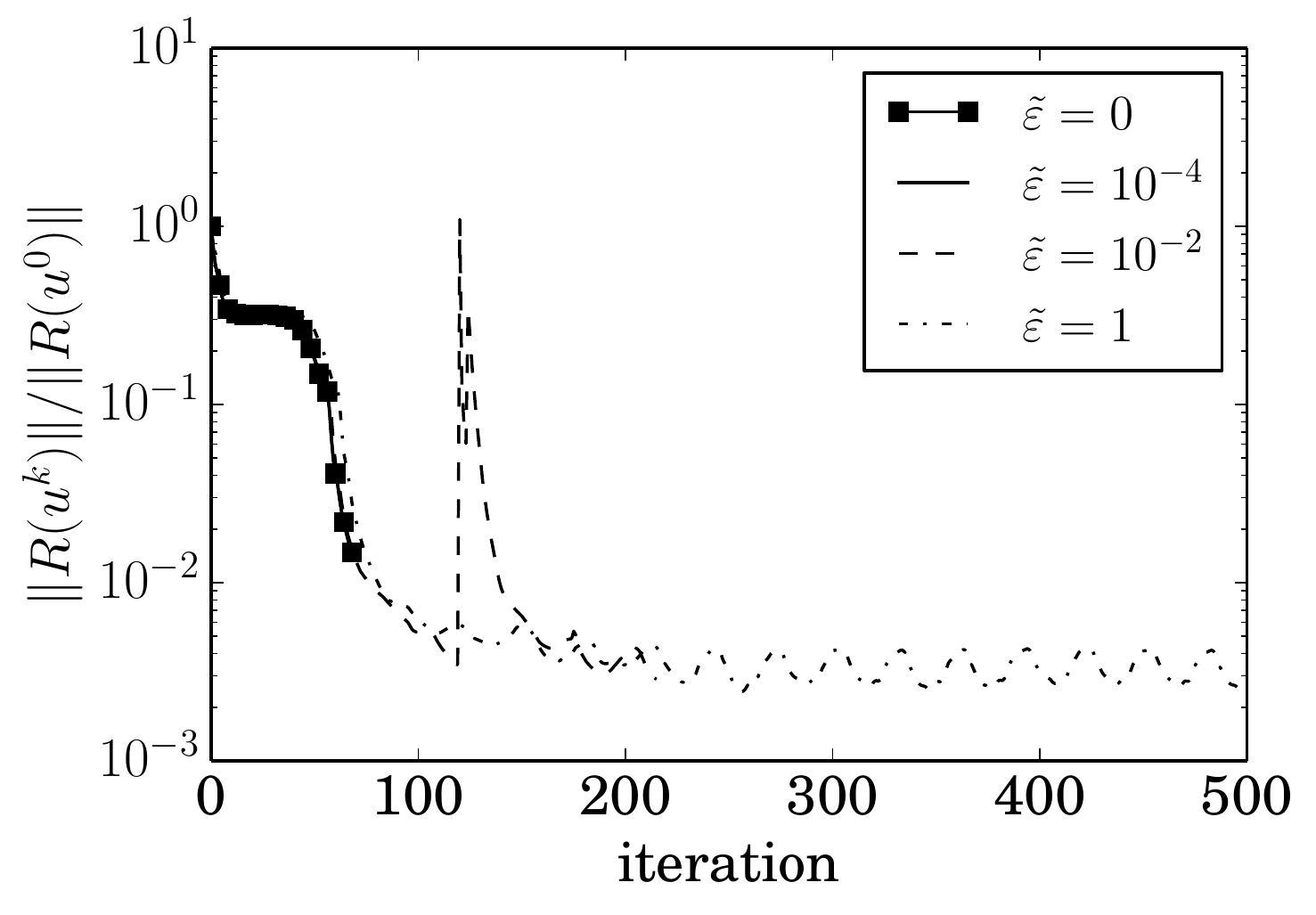}
		\caption{$q=5$.}
		\label{sfig.scramjetq5-1}
	\end{subfigure}
	\begin{subfigure}[b]{0.45\textwidth}
		\includegraphics[width=\textwidth]{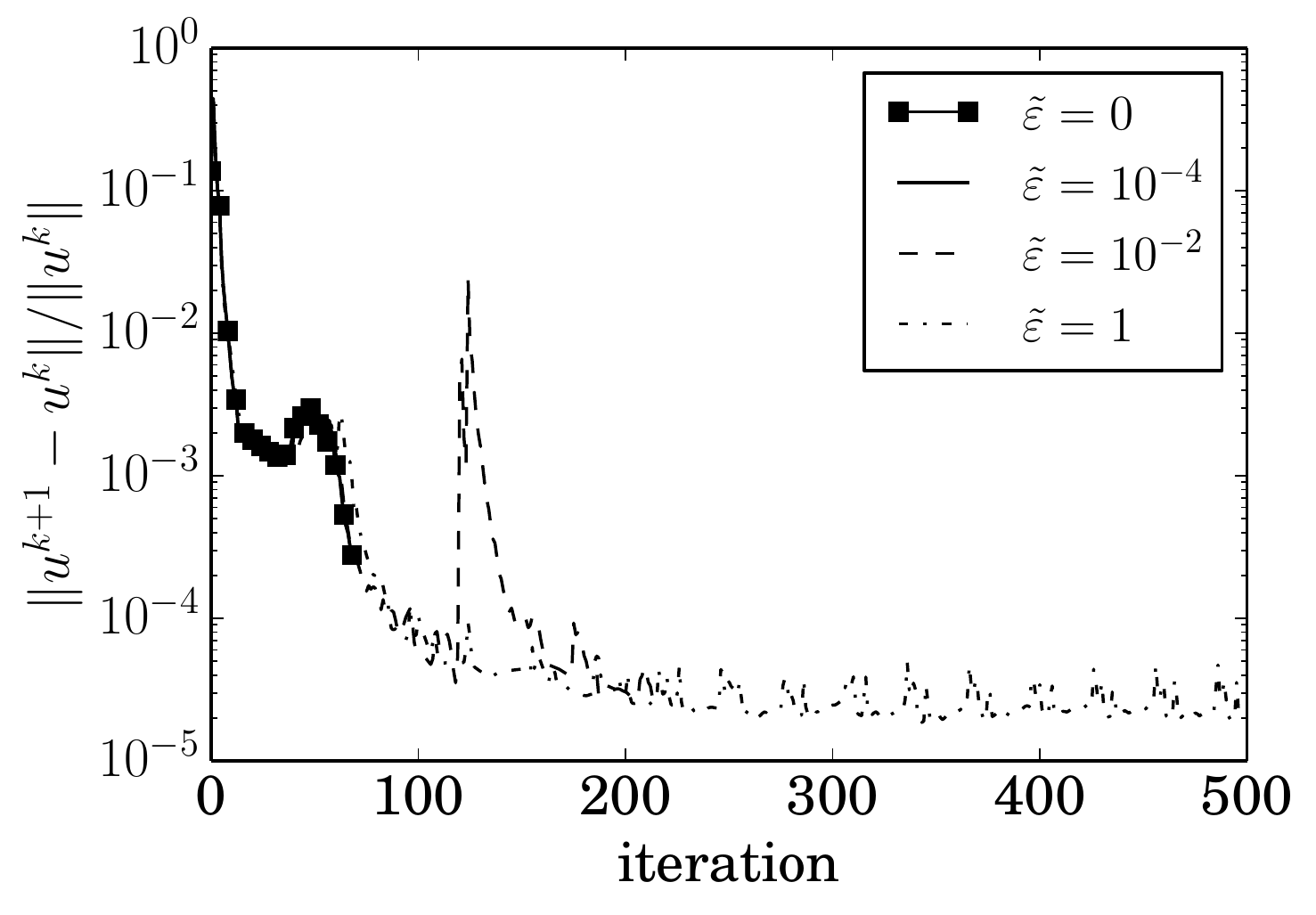}
		\caption{$q=5$.}
		\label{sfig.scramjetq5-2}
	\end{subfigure}
	\caption{Comparison of the convergence behavior for the Scramjet test and different regularization parameters choices. A coarse mesh of 18476 $\mathcal{Q}_1$ elements is used.}
	\label{fig.scramjet-coarse}
\end{figure}

\begin{figure}[h]
	\centering
	\begin{subfigure}[b]{0.45\textwidth}
		\includegraphics[width=\textwidth]{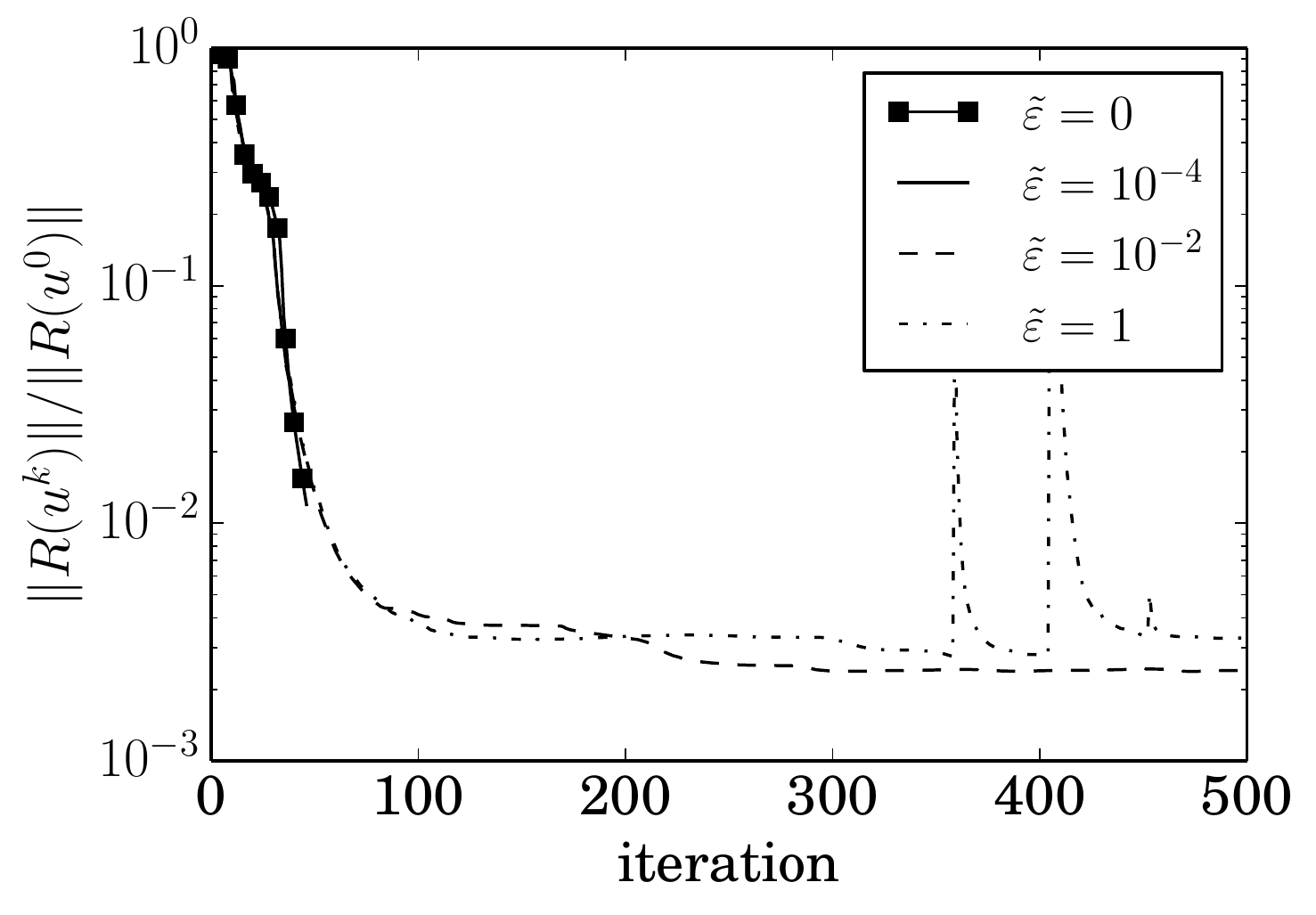}
		\caption{$q=2$.}
		\label{sfig.scramjetq2-3}
	\end{subfigure}
	\begin{subfigure}[b]{0.45\textwidth}
		\includegraphics[width=\textwidth]{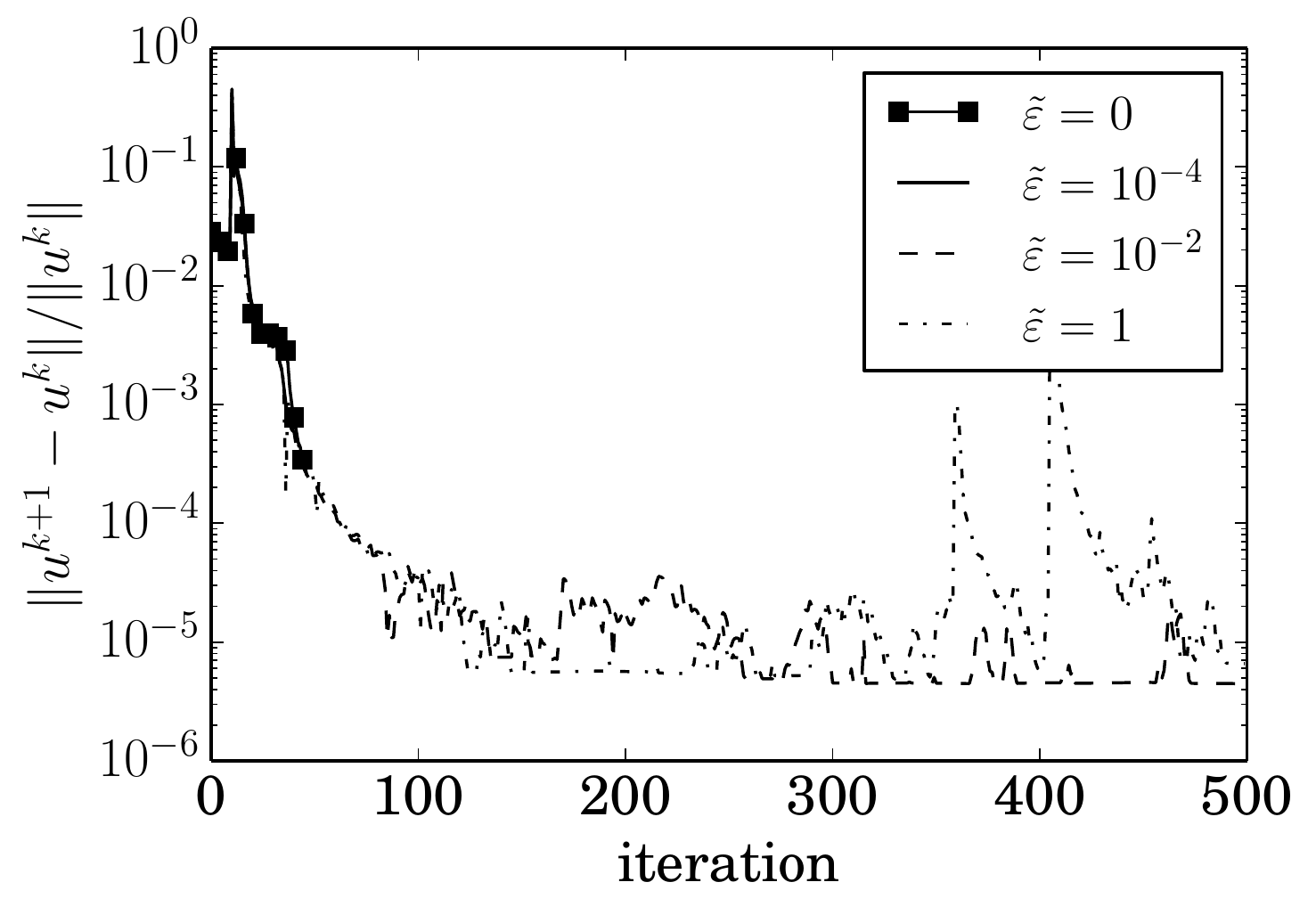}
		\caption{$q=2$.}
		\label{sfig.scramjetq2-4}
	\end{subfigure}
	\begin{subfigure}[b]{0.45\textwidth}
		\includegraphics[width=\textwidth]{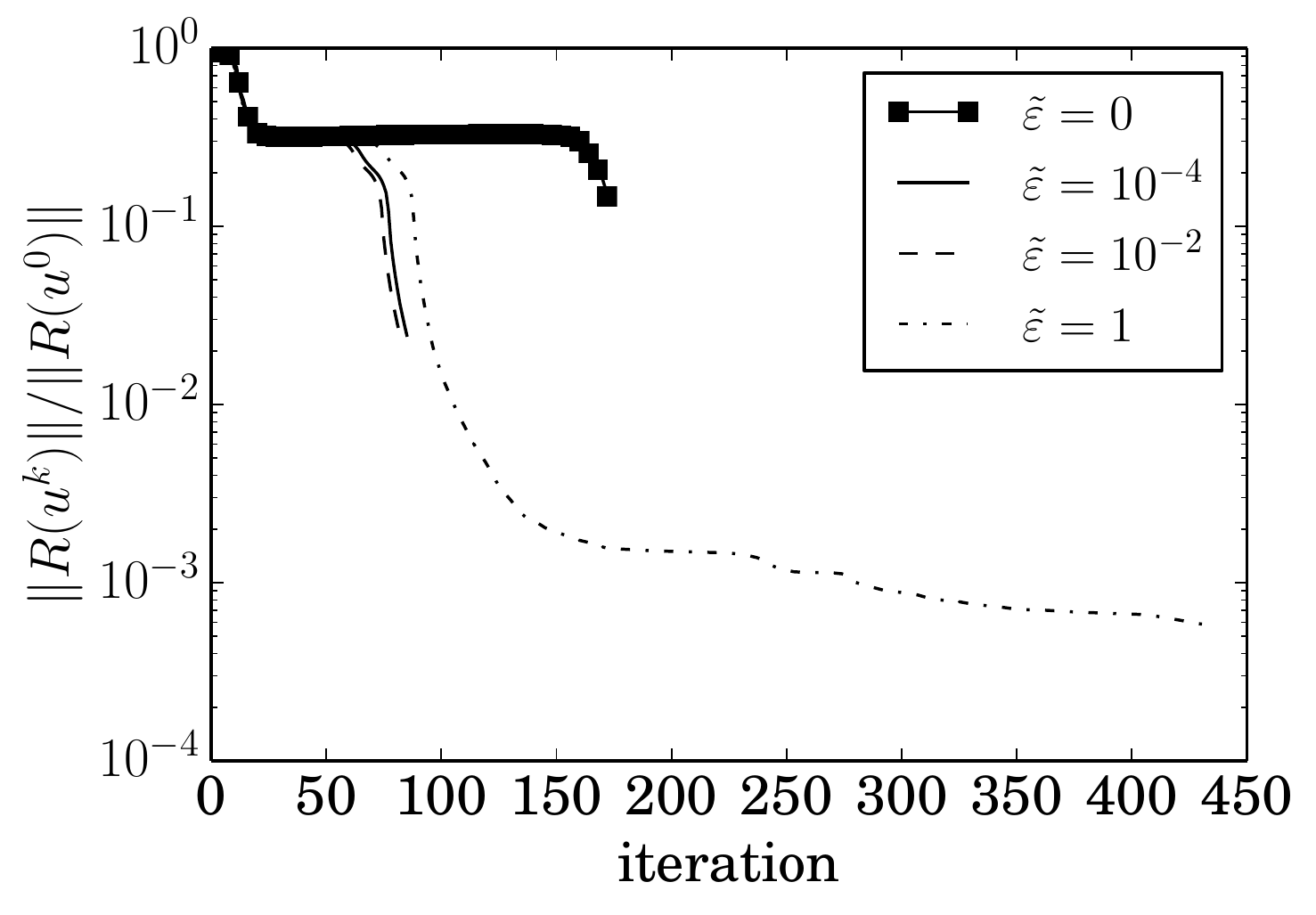}
		\caption{$q=5$.}
		\label{sfig.scramjetq5-3}
	\end{subfigure}
	\begin{subfigure}[b]{0.45\textwidth}
		\includegraphics[width=\textwidth]{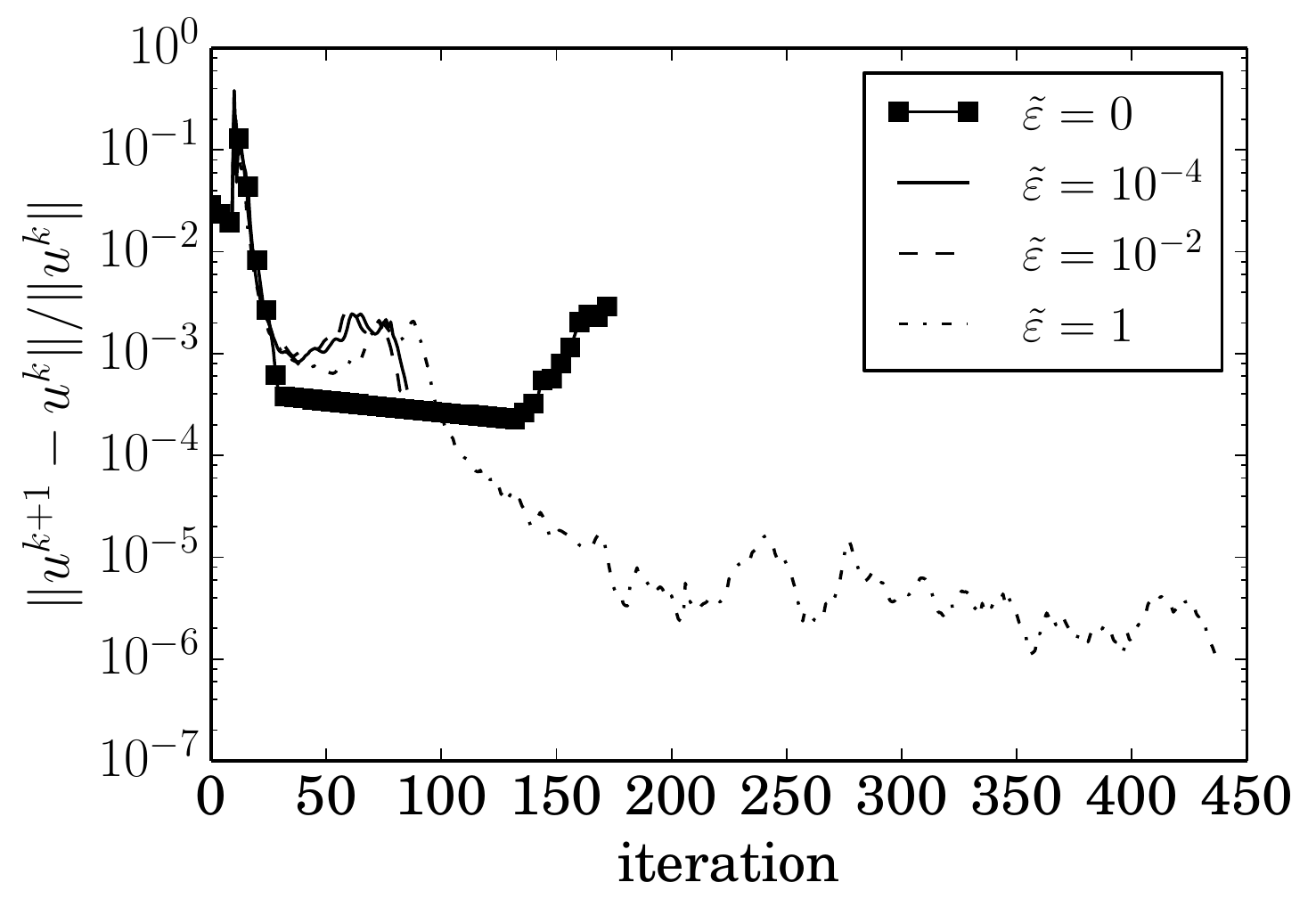}
		\caption{$q=5$.}
		\label{sfig.scramjetq5-4}
	\end{subfigure}
	\caption{Comparison of the convergence behavior for the Scramjet test and different regularization parameters choices. A fine mesh of 63695 $\mathcal{Q}_1$ elements is used.}
	\label{fig.scramjet-fine}
\end{figure}

\section{Conclusions}\label{sec.conclusions}
In this work, a differentiable 
\emph{local bounds preserving} stabilization 
for Euler equations was presented. 
This stabilization is based on the 
combination of a differentiable shock 
detector, a partially lumped mass matrix, 
and Rusanov artificial diffusion. 
The scheme has been successfully 
tested in steady and transient 
benchmark problems. Numerical results 
show that the proposed method exhibits 
good stability properties. Application
of the scheme to steady and transient 
problems with shocks resulted in well 
resolved profiles.
In addition the differentiable shock detector, 
a continuation method for the regularization 
parameters in the differentiable 
stabilization was presented. This was done 
to improve nonlinear convergence. Nonlinear 
convergence of the scheme was then analyzed 
for the differentiable version was compared 
to that of the non-regularized counterpart. 
The differentiable stabilization showed 
better convergence, especially when the 
hybrid Picard--Newton method was used.  
For small steady problems, the scheme is 
able to converge directly to the steady 
state solution without making use of 
pseudo-transient time stepping. However, 
for problems with complex shock patterns 
the scheme only converges to moderate 
tolerances. Numerical results also show that 
differentiability not only improves
nonlinear convergence, but also improves 
the robustness of the method. 
In the case of transient problems, 
some improvement in the computational 
cost is observed. However, since the 
non-differentiable method already 
exhibits good nonlinear convergence, 
there is not much room for improvement. 
Nevertheless, it is possible to show 
that the differentiable stabilization 
can achieve a similar accuracy while 
lowering the computational cost.

\section*{Acknowledgments}
The work of S. Mabuza and J.N. Shadid was 
partially supported by the U.S. Department 
of Energy, Office of Science, Office of 
Applied Scientific Computing Research. 
Sandia National Laboratories is a multi-mission 
laboratory managed and operated by 
National Technology and Engineering 
Solutions of Sandia, LLC., a wholly 
owned subsidiary of Honeywell 
International, Inc., for the U.S. 
Department of Energy's National Nuclear 
Security Administration under contract 
DE-NA0003525. This paper describes 
objective technical results and analysis. 
Any subjective views or opinions that 
might be expressed in the paper do not 
necessarily represent the views of the 
U.S. Department of Energy or the United 
States Government. J. Bonilla gratefully 
acknowledges the support received from 
''la Caixa'' Foundation through
its PhD scholarship program 
(LCF/BQ/DE15/10360010). S. Badia gratefully 
acknowledges the support received 
from the Catalan Government through the 
ICREA Acad\`emia Research Program. 
S. Badia and J. Bonilla also acknowledge 
the financial support to CIMNE via the 
CERCA Programme  / Generalitat de Catalunya.

\bibliographystyle{siam}
\bibliography{art037}

\end{document}